\documentclass[11pt,a4paper]{article}

\usepackage[T1]{fontenc}
\usepackage{lmodern}
\usepackage[utf8]{inputenc}

\usepackage[left=2.45cm, top=2.45cm,bottom=2.45cm,right=2.45cm]{geometry}
\usepackage{amsmath}
\usepackage{amssymb,amsthm,graphicx,xcolor,tikz}
\usepackage[british]{babel}
\usepackage{bbm}

\usepackage[
	bookmarks=true,
	bookmarksnumbered=true,
	bookmarksopen=true,
	unicode=true,
	pdftoolbar=true,
	pdfmenubar=true,
	pdffitwindow=false,
	pdfstartview={FitH},
	pdftitle={},
	pdfauthor={},
	pdfsubject={},
	pdfcreator={},
	pdfproducer={},
	pdfkeywords={},
	pdfnewwindow=true,
	colorlinks=true,
	linkcolor=black,
	citecolor=black,
	filecolor=black,
	urlcolor=black]{hyperref}

\makeatletter
	\def\@cite#1#2{[\textbf{#1}\if@tempswa , #2\fi]}	%text
	\def\@biblabel#1{[#1]}								%bibliography
\makeatother
\numberwithin{equation}{section}
\newtheorem {theorem}{Theorem}[section]

\newtheorem {lemma}[theorem]{Lemma}

\theoremstyle{definition}
\newtheorem{definition}[theorem]{Definition}
\newtheorem {remark}[theorem]{Remark}
\newtheorem {example}[theorem]{Example}

\newcommand{\EE}{\mathbb{E}}

\newcommand{\NN}{\mathbb{N}}
\newcommand{\PP}{\mathbb{P}}
\newcommand{\RR}{\mathbb{R}}

\DeclareMathOperator{\Vol}{Vol}

\DeclareMathOperator{\Var}{Var}

\DeclareMathOperator{\Kol}{Kol}

\DeclareMathOperator{\conv}{conv}

\DeclareMathOperator{\bd}{bd}

\DeclareMathOperator{\Vis}{Vis} 

\newcommand{\dint}{\mathrm{d}}
\newcommand{\id}{\mathbf 1}   

\usepackage{accents}
\DeclareMathSymbol{\widetildesym}{\mathord}{largesymbols}{"65}

\newcommand{\tilf}{\tilde{f}}
\newcommand{\tV}{ \widetilde{V} }

\makeatletter
\let\@fnsymbol\@alph
\makeatother

\begin{document}

%\title\textbf{series Weighted random inscribed polytopes\\ with applications}
%\title\textbf{series Weighted random inscribed polytopes: \\ projective Finsler metrics, dual Brunn-Minkowski theory and polyhedral approximation}
%\title\textbf{series Weighted random inscribed polytopes}
\title{\bfseries Random Inscribed Polytopes\\ in Projective Geometries}
\author{%
    Florian Besau\footnotemark[1]%
    \and Daniel Rosen\footnotemark[2]%
    \and Christoph Th\"ale\footnotemark[3]%
}

\date{}
\renewcommand{\thefootnote}{\fnsymbol{footnote}}
\footnotetext[1]{%
    Vienna University of Technology, Austria. Email: florian.besau@tuwien.ac.at
}

\footnotetext[2]{%
    Ruhr University Bochum, Germany. Email: daniel.rosen@rub.de
}

\footnotetext[3]{%
    Ruhr University Bochum, Germany. Email: christoph.thaele@rub.de
}

\maketitle

\begin{abstract}\noindent
We establish central limit theorems for natural volumes of random inscribed polytopes in projective Riemannian or Finsler geometries. In addition, normal approximation of dual volumes and the mean width of random polyhedral sets are obtained. We deduce these results by proving a general central limit theorem for the weighted volume of the convex hull of random points chosen from the boundary of a smooth convex body according to a positive and continuous density in Euclidean space. In the background are geometric estimates for weighted surface bodies and Berry--Esseen bounds for functionals of independent random variables.

    \smallskip\noindent
    \textbf{Keywords.} Central limit theorem, dual Brunn--Minkowski theory, hyperbolic geometry, inscribed polytopes, non-Euclidean geometry, projective Finsler geometry, random polytope, spherical geometry, surface body, weighted random polytope.

    \smallskip\noindent
    \textbf{MSC 2010.} Primary  52A22, 52A55; Secondary 58B20, 60D05, 60F05.
\end{abstract}

\section{Introduction and main results}

\subsection{Background}
The theory of random convex hulls has a long history, going back to Sylvester's famous four-point problem \cite{Syl}. Since the seminal papers of R\'{e}yni and Sulanke \cite{RS63,RS64}, it has become a mainstream research topic in convex, stochastic and integral geometry, with connections to asymptotic geometric analysis, optimization or multivariate statistics, to name just a few. 

In this article we focus on the convex hull of independent and identically distributed points taken from the boundary of a fixed convex body $K$. This model of a so-called \emph{random inscribed polytope} in $K$ was investigated in \cite{BFH,BMT,Rei02,Rei03,RVW07,RVW,SchWer, TurWes}, mainly from an asymptotic point of view (as the number of points tends to infinity). In particular, in \cite{Tha} a central limit theorem is proven for the volume of the random inscribed polytope inside a sufficiently smooth convex body in Euclidean space. Our goal is to generalize this result to the setting of non-Euclidean geometries. Of particular interest are the cases of random inscribed polytopes inside convex bodies in spherical or hyperbolic geometry. This continues a recent trend in stochastic geometry of generalizing known results to the non-Euclidean setting, 
and in particular to spherical and hyperbolic geometry, see e.g.\ \cite{BHRS,BHPS,DHT,GodlandKabluchko,HHT,HugRei,HugSchneider,HugSchneiderThreshold,HugTha, KabluchkoHalfSphere,KabluchkoThaele19,KabluchkoThaele20}.

More generally, we work with \emph{projective} Finsler geometries, i.e., ones for which geodesics are affine line segments. These are the Finsler solutions of Hilbert's fourth problem, and have been studied intensively, see e.g.\ \cite{AP05,Bus74,Papa, Pog}. Since on a Finsler manifold there is no canonical volume measure, we establish our results for a general \emph{definition of volume}, which is an assignment of Finsler volume measure obeying some natural axioms \cite{APT}. 

\smallskip 
Following the ideas of \cite{BLW}, we reformulate the problem in terms of \emph{weighted} random inscribed polytopes in Euclidean space. This approach is more general and also paves the way to some new directions. For example, it allows us to prove central limit theorems for dual volumes, which are central to Lutwak's dual Brunn--Minkowski theory \cite{Lut75,Lut79}, as well as for the mean width of random polyhedral sets circumscribing a convex body. Finally, let us also mention that the analogous result for the random model where points are distributed inside the convex body was proven for the Euclidean case in \cite{Reitzner:2005}, and were recently generalized to the non-Euclidean setting in \cite{ThaBes}. 
%Against this light our paper can be regarded as a natural continuation of \cite{ThaBes}.

\subsection{Random inscribed polytopes in projective Riemannian geometries}\label{subsec:Riemannian}

We begin with the setting of Riemannian geometry. In this case there is a canonical notion of volume--the \emph{Riemannian volume measure}. It may be defined as the $d$-dimensional Hausdorff measure of the associated metric space of the $d$-dimensional Riemannian manifold $(\Omega,g)$, or equivalently, as the integral of the Riemannian volume density, which in local coordinates reads
\begin{equation}\label{eq:riem-vol-density}
    \sqrt{\det(g_{ij}(x))} \,|dx_1 \wedge \ldots  \wedge dx_d|,
\end{equation} where $g_{ij}(x)$ are the metric coefficients in the given coordinates (see, e.g., \cite[\S 5.5.1]{BBI}).

A Riemannian metric on a domain $\Omega\subset\RR^d$ is called \emph{projective} if affine line segments are geodesics. 
We consider a projective $C^2$-Riemannian metric $g$ on a convex domain $\Omega \subset \RR^d$. The regularity assumption ensures uniqueness of geodesics, so in particular affine line segments are the only geodesics of $g$.

% Let $\Omega \subset \RR^d$ be a convex domain, equipped with a $C^2$-Riemannian metric $g$. We assume that $g$ is \emph{projective}, that is, straight line segments in $\Omega$ are geodesics of $g$. We recall that in Riemannian geometry there is a canonical notion of volume, which may be defined as the Hausdorff measure of the associated metric space, or equivalently as the integral of the Riemannian volume density, which in local coordinates reads 
%\begin{equation}\label{eq:riem-vol-density}
%\sqrt{\det(g_{ij}(x))} \,|dx_1 \wedge \cdots  \wedge dx_n|,
%\end{equation} where $g_{ij}(x)$ are the metric coefficients in the given coordinates (see \cite[\S 5.5.1]{BBI}, which can also be deduced from \cite[Theorem 2.10.10]{Fed}).

Let $K \subset \Omega$ be a convex body of class $C^2_+$, that is, the boundary $\bd K$ of $K$ is a $C^2$-smooth hypersurface with everywhere strictly positive Gauss--Kronecker curvature. Denote by  $\Phi_g$ the Riemannian volume measure on $K$, and by  $\sigma_g$ the normalized Riemannian surface measure on $\bd K$. Let $X_1,X_2,\ldots$ be a sequence of independent random points on $\bd K$ distributed according to $\sigma_g$ and for $n\geq d+1$ define their convex hull $K_g(n):=[X_1, \ldots, X_n]$, which is what we call a \emph{random Riemannian inscribed polytope}.

\begin{theorem}\label{cor:Riemann}
%    Let $\Omega\subset \RR^d$ be an open and convex domain, and $g$ be a projective $C^2$-smooth Riemannian metric on $\Omega$. Furthermore, let $K\subset \Omega$ be a convex body of class $C^2_+$.
Under the above assumptions, the Riemannian volume $\Phi_g(K_g(n))$ satisfies a central limit theorem, that is,
	\begin{equation*}
	\frac{\Phi_g(K_g(n)) - \EE \Phi_g(K_g(n)) }{\sqrt{\Var \Phi_g(K_g(n)) }} \overset{d}{\longrightarrow} Z \qquad \text{as $n\to \infty$},
	\end{equation*}
	where $Z$ is a standard Gaussian random variable. Here $\overset{d}{\longrightarrow}$ denotes convergence in distribution.
\end{theorem}

\begin{example}[Hyperbolic geometry]
\begin{figure}[t]
    \centering
    \begin{tikzpicture}
        \clip (-5,-0.5) rectangle (5,4.5);
        \node at (0,0) {\includegraphics[width=0.8\textwidth]{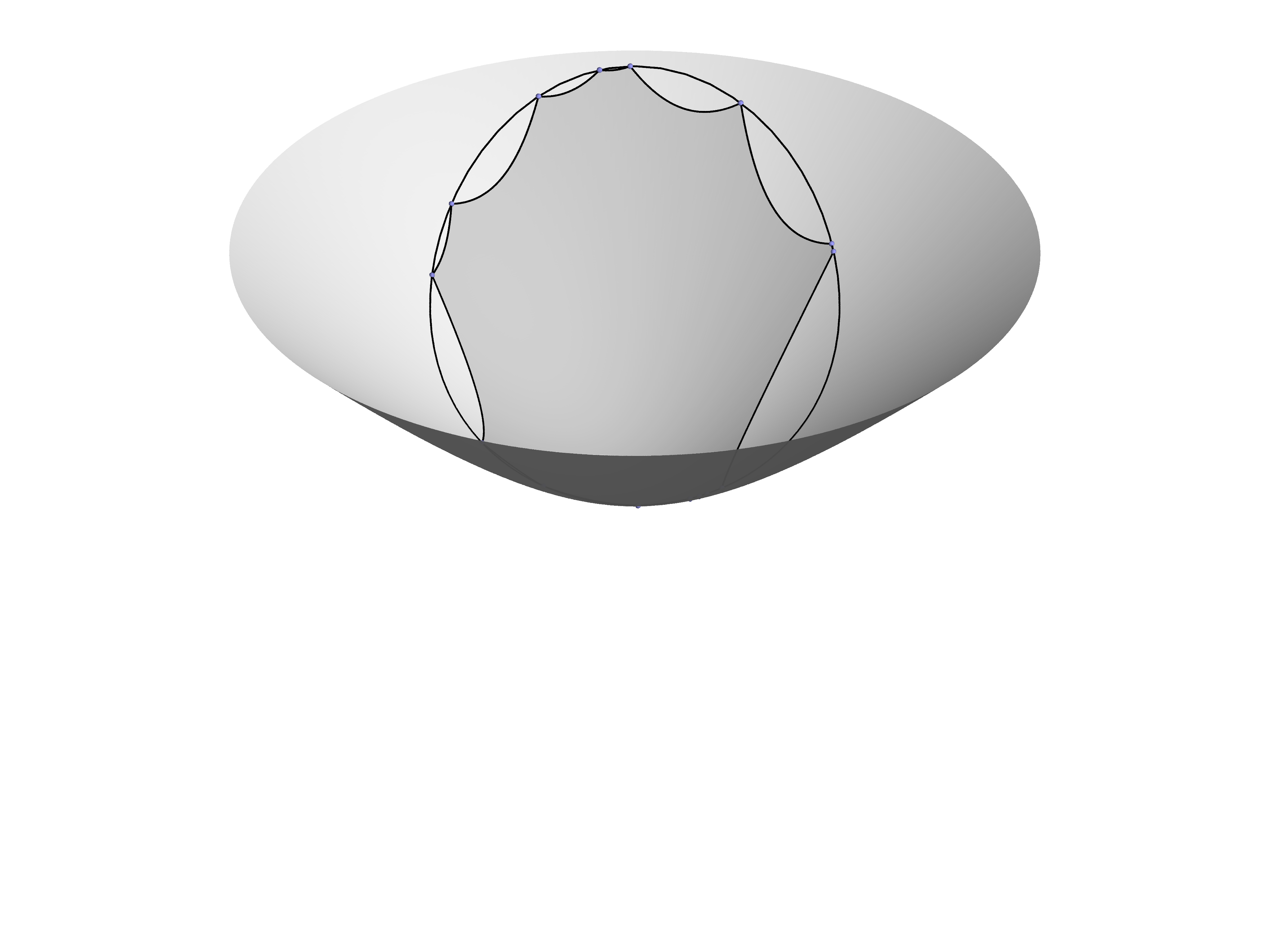}};
        \node at (2.5,2) {$K$};
        \node at (0 ,3) {$K_h(n)$};
        %\draw (-5,-0.5) rectangle (5,4.5);
    \end{tikzpicture}
    \caption{Illustration of the hyperbolic inscribed random polytope $K_h(n)$ generated in a hyperbolic convex body $K$
    in the hyperbolid model of the hyperbolic plane.}
    \label{fig:hyperbolid}
\end{figure}
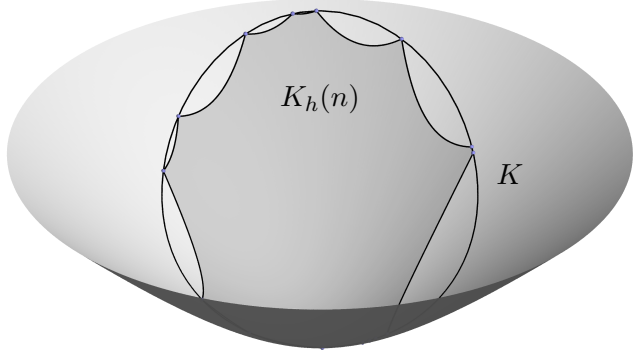
    The $d$-dimensional hyperbolic space is realized as a projective Riemannian space in the \emph{Beltrami--Klein} model.
%	 This is the unit disc $\Omega = \{x \in \RR^d \,:\, \sum_{j=1}^d x_j^2 < 1 \}$ equipped with the Riemannian metric with length element
%	\[ds^2 =   \frac{(1-\sum_{j=1}^d x_j^2 )\sum_j dx_j^2 + \left(\sum_{j=1}^d x_j dx_j \right)^2 }{(1-\sum_{j=1}^d x_j^2 )^2}.\]
 This is the unit disc $\Omega = \{x \in \RR^d \,:\, \|x\| <1  \}$ equipped with the Riemannian metric with length element
\[ds^2 =   \frac{(1-\|x\|^2 ) \|dx\|^2 + \langle x, dx \rangle^2 }{(1-\| x\|^2 )^2},\]
where we denote by $\langle \,\cdot\,, \,\cdot\, \rangle$
 and $\|\,\cdot\,\|$ the Euclidean inner product and norm, respectively, on $\RR^d$. This defines a (complete) projective Riemannian metric of constant sectional curvature $-1$ (see e.g. \cite{AVS, CFKP} for more details, and relations with other models of hyperbolic space). Then Theorem \ref{cor:Riemann} implies that the hyperbolic volume of the random polytope generated by independent points on the boundary of a hyperbolic convex body obeys a central limit theorem. Figure \ref{fig:hyperbolid} illustrates a random inscribed polytope in the hyperboloid model of the hyperbolic plane.

\end{example}

\begin{example}[Spherical geometry]
	The spherical geometry in a hemisphere may also be realized in a projective model. \emph{The gnomonic projection} maps the upper hemisphere $S^d_+ := \{x \in S^{d} : x_{d+1} > 0\} \subset \RR^{d+1}$ onto its tangent hyperplane at the north pole, $H:=\{x_{d+1}=1\}$, by projecting along rays emanating from the origin (see Figure \ref{fig:gnomonic}). To be more precise, the point $x=(x_1,\dotsc, x_{d+1}) \in S^d_+$ is mapped to the point $p(x)=(\frac{x_1}{x_{d+1}}, \dotsc, \frac{x_d}{x_{d+1}}) \in \RR^d$, where we have identified $H$ with $ \RR^d$ by means of an isometry mapping the north pole of $S^d$ to the origin of $\RR^d$. The standard Riemannian metric on the hemisphere $S^d_+$ is identified with the Riemannian metric on $\RR^d$ with length element
	\[
	ds^2 =  \frac{(1+\|x\|^2 ) \|dx\|^2 - \langle x, dx \rangle^2 }{(1+\| x\|^2 )^2}.
	\]	
	\begin{figure}[t]
	\centering
	\begin{tikzpicture}
	\clip (-3.5,-2.5) rectangle (3.5,2.7);
	\node at (0,0) {\includegraphics[width=7cm]{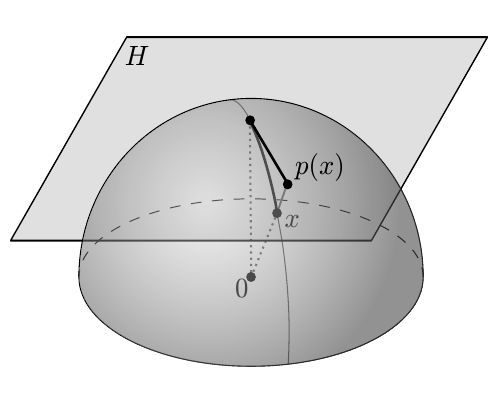}};
	%\draw (-3.5,-2.5) rectangle (3.5,2.7);
	\end{tikzpicture}
	\caption{The gnomonic (central) projection from the upper hemisphere.}\label{fig:gnomonic}
\end{figure}
	This defines a projective Riemannian metric on $\Omega = \RR^d$ with constant sectional curvature $+1$. Theorem \ref{cor:Riemann} then implies that the spherical  volume of the random polytope generated by independent points on the boundary of a spherical convex body obeys a central limit theorem. Figure \ref{fig:sphere} illustrates a random inscribed polytope on the upper hemisphere $S^2_+$.
\begin{figure}[t]
    \centering
    \begin{tikzpicture}
        \clip (-4,-1.4) rectangle (4,4);
        \node at (0,0) {\includegraphics[width=0.8\textwidth]{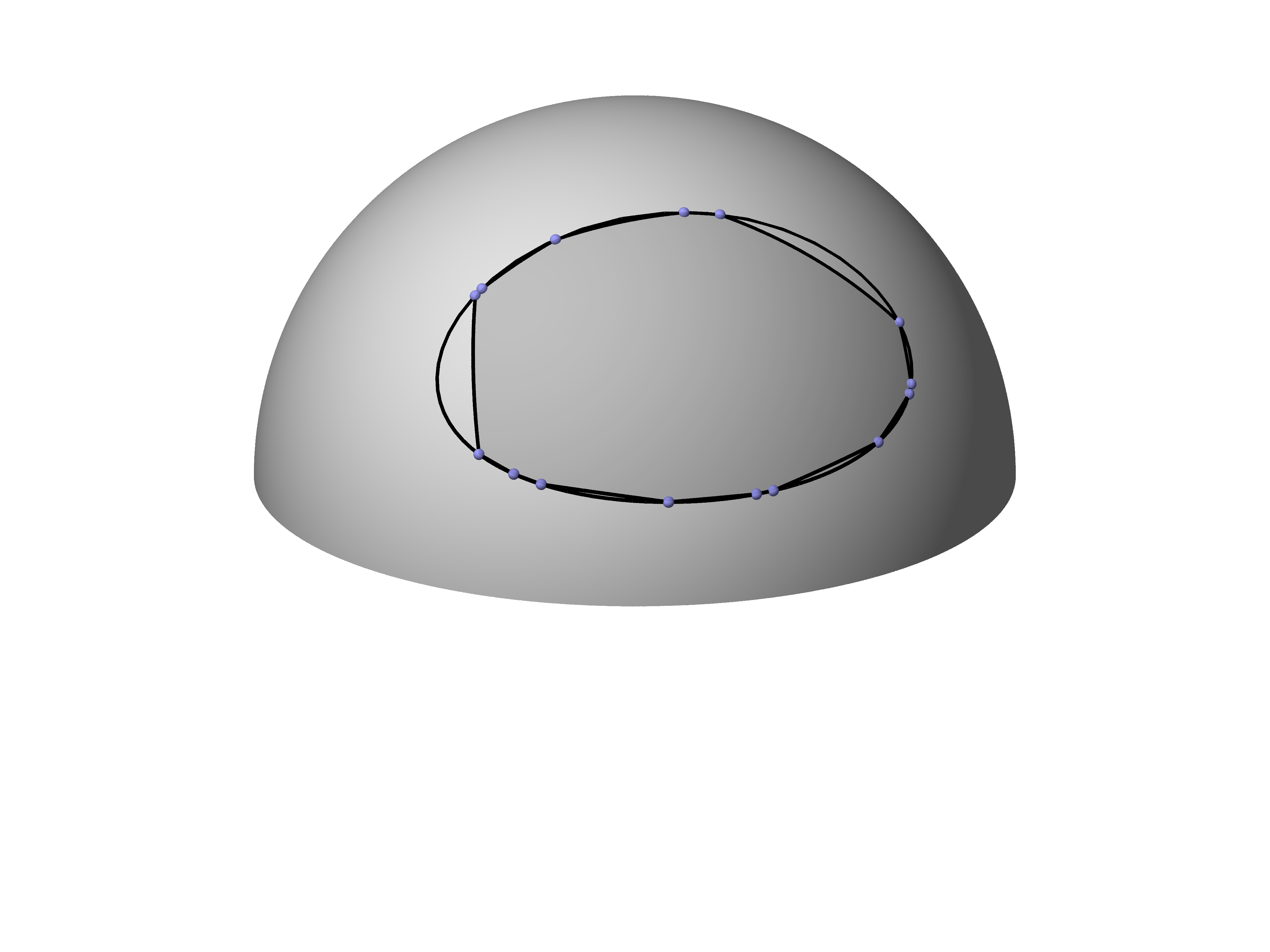}};
        \node at (-2.4,1) {$K$};
        \node at (0.8 ,1.8) {$K_s(n)$};
        %\draw (-4,-1.4) rectangle (4,4);
    \end{tikzpicture}
    \caption{Illustration of the spherical inscribed random polytope $K_s(n)$ generated in a spherical convex body $K$ contained in the open hemisphere $S^2_+$.}
    \label{fig:sphere}
\end{figure}
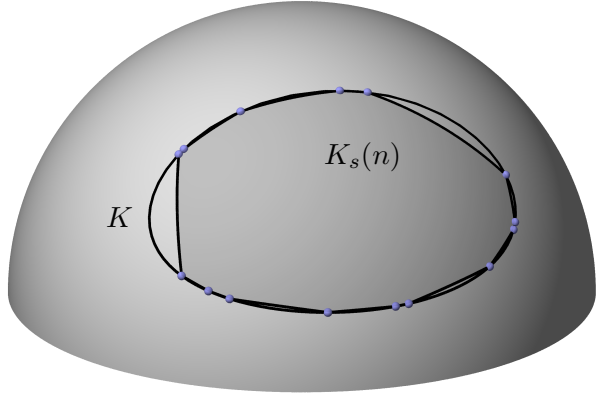

\end{example}

\begin{remark}
	The classical Beltrami theorem states that any projective Riemannian metric on a convex domain $\Omega \subset \RR^d$ is of constant sectional curvature, and hence locally isometric to (a rescaling of) either the Euclidean, hyperbolic, or spherical space (see, e.g., \cite{BesauWerner:2018,Bus, Mat,Sha}). Moreover, if the metric is complete and the underlying space simply connected, it is \emph{globally} isometric to one of these spaces.
\end{remark}

\subsection{Random inscribed polytopes in projective Finsler geometries}

We turn now to the more general case of projective Finsler metrics. For this we let $\Omega \subset \RR^d$ be a convex body, equipped with a Finsler metric $F$, i.e., a continuous function $F:T\Omega \to \RR$ on the tangent bundle $T\Omega$ of $\Omega$ such that for all $x \in \Omega$, $F(x,\, \cdot\,) :T_x\Omega \to \RR$ is a norm, where we write $T_x\Omega$ for the tangent space of $\Omega$ at $x$. We assume that $F$ is $C^3$-smooth away from the zero section of $T\Omega$ and is strongly convex, that is, the vertical Hessian $\frac{\partial^2 F^2}{\partial v_i \partial v_j}(x,v)$ is non-degenerate at every $v \neq 0$. We assume moreover that $F$ is \emph{projective}, that is, straight line segments are geodesics of $F$. Again, the regularity assumption on $F$ implies that these are the only geodesics.

We note that in Finsler geometry, unlike Riemannian geometry, there does not exist a canonical choice of volume measurement. However, any 'reasonable' notion of Finsler volume is completely determined by its value on normed spaces (see e.g.\ \cite[\S 5.5.3]{BBI}). A Lebesgue measure on a $d$-dimensional normed space $X$ can be described in terms of a (positive) density, that is, a norm on the ($1$-dimensional) top exterior power $\bigwedge^dX$. This leads to the following axiomatic definition due to \'{A}lvarez Paiva and Thompson \cite{APT}. 

\begin{definition}
A \emph{definition of volume} on $d$-dimensional normed spaces is an assignment to each $d$-dimensional normed space $X$ of a norm $\mu_X$ on $\bigwedge^d X$ such that the following conditions are satisfied:
\begin{enumerate}
	\item If $T:X \to Y$ is a short map (i.e., a linear map of norm $\leq 1$), the induced map $\bigwedge^d T : \bigwedge^dX \to \bigwedge^d Y$ is short as well.
	\item The assignment $X \mapsto (\bigwedge^d X, \mu_X)$ is continuous in the Banach--Mazur topology.
	\item If $X$ is a Euclidean space, $\mu_X$ is the standard Euclidean volume measure.
\end{enumerate}
\end{definition}

Given a definition of volume on a $d$-dimensional normed space, one can define a volume on a general $d$-dimensional Finsler manifolds, by the following procedure. If $(M,F)$ is a Finsler manifold, that is, a differentiable manifold $M$ together with a Finsler metric $F$ on $TM$, then for each $x \in M$ we obtain a norm $\mu_{T_xM}$ on $\bigwedge^dT_xM$. This norm varies continuously with $x$, and hence defines a continuous \emph{volume density} on $M$. Volume densities can be integrated (see e.g. \cite{BT,Nico}), yielding a volume measure on $M$. 

The following examples are taken from \cite{APT}.

\begin{example}[The Busemann definition \cite{Bus47}]
	The Busemann definition of volume of a $d$-dimensional normed space $X$ is such that the volume of the unit ball of $X$ is ${\rm Vol}_{\rm Bus}(B) = \kappa_d$, where $\kappa_d$ is the volume of the $d$-dimensional Euclidean unit ball. The corresponding density on $X$ is given by
	\begin{equation*}
	\mu_{\rm Bus}(v_1 \wedge \cdots \wedge v_d)= \frac{\kappa_d}{{\rm Vol}(B;v_1, \ldots, v_n)},
	\end{equation*}
	where ${\rm Vol}(B;v_1, \ldots, v_n)$ denotes the volume of $B$ with respect to the Lebesgue measure determined by the basis $v_1, \ldots v_d$. The resulting volume measure on a $d$-dimensional continuous Finsler manifold is known to coincide with its $d$-dimensional Hausdorff measure (see \cite[\S 6]{Bus47}.)
\end{example}

\begin{example}[The Holmes--Thompson definition \cite{HT}]
	The Holmes--Thompson volume definition uses the canonical symplectic structure on $X \times X^*$ (see e.g. \cite{APT, MS}), and the associated symplectic volume. The Holmes--Thompson volume of the unit ball $B$ of $X$ is equal to the symplectic volume of $B \times B^* \subset X \times X^*$ divided by $\kappa_d$. The corresponding density is given by 
	\begin{equation*}
	\mu_{\rm HT}(v_1 \wedge \cdots \wedge v_n) = \frac{ {\rm Vol}(B^*; \xi_1, \ldots, \xi_d) }{\kappa_d},
	\end{equation*}
	where $\xi_1, \ldots, \xi_d$ is the basis of $X^*$ dual to $v_1, \ldots,v_d$. It is known that the resulting Holmes--Thompson volume of a $d$-dimensional Finsler manifold is equal to the symplectic volume of the unit co-disc bundle $B^*(M) \subset T^*M$ with respect to the canonical symplectic structure on $T^*M$, divided by $\kappa_d$ (see e.g. \cite{APT, MS}).
\end{example}

\begin{example}[The Gromov mass and mass$^*$ definitions \cite{Gro}]
	The Gromov mass definition is such that the maximal cross-polytope inscribed in the unit ball of $X$ has volume $2^n/n!$. The corresponding density is given by 
	\begin{equation*}
	\mu_{\rm mass}(a) = \inf \|v_1\| \cdots \|v_d\|,
	\end{equation*}
	where the infimum extends over all $v_1, \ldots, v_d$ such that $a= v_1 \wedge \cdots \wedge v_d$.
	
	The dual notion is the Gromov mass$^*$ definition, for which the minimal parallelotope circumscribed about the unit ball of $X$ has volume $2^d$. The corresponding density is given by 
	\begin{equation*}
	\mu_{\rm{mass}^* }(v_1 \wedge \cdots \wedge v_d) = \left[\mu_{\rm{mass}} (\xi_1 \wedge \cdots \wedge \xi_d) \right]^{-1},
	\end{equation*}
	where $\xi_1,\ldots, \xi_d$ is the dual basis to $v_1, \ldots, v_d$, and the mass definition on the right hand side is applied to the dual space of $X$. 
\end{example}

\medskip
We now return to our setting of a projective Finsler metric on a convex domain $\Omega$. We fix a definition of volume on $d$-dimensional normed spaces, which defines a volume measure on $\Omega$, as we explained above. We denote this volume measure by $\Phi$. Now, given a convex body $K \subset \Omega$ of class $C^2_+$, fixing another definition of volume on $(d-1)$-dimensional normed spaces defines a surface measure on $\bd K$, and we denote the resulting normalized probability measure on $\bd K$ by $\sigma$. Let $X_1,X_2, \ldots$ be a sequence of independent random points on $\bd K$ distributed according to $\sigma$. For $n\geq d+1$ the convex hull $K_F(n):=[X_1, \ldots, X_n]$ is called the \emph{random inscribed Finsler polytope}.

\begin{theorem}\label{thm:Finsler}
    Let $\Omega\subset \RR^d$ be an open and convex domain, and $F$ be a projective Finsler metric on $\Omega$ that is strongly convex and $C^3$-smooth away from the zero section of $T\Omega$.
	Then the Finsler volume $\Phi(K_F(n))$ of the random Finsler polytope $K_F(n)$ satisfies a central limit theorem, that is,
	\begin{equation*}
        \frac{\Phi(K_F(n)) - \EE \Phi(K_F(n)) }{\sqrt{\Var \Phi(K_F(n))}} \overset{d}\longrightarrow Z \qquad \text{as $n\to \infty$},
	\end{equation*}
	where $Z$ is a standard Gaussian random variable.
\end{theorem}

\begin{remark}
    Theorem \ref{cor:Riemann} is almost a special case of Theorem \ref{thm:Finsler}, except it allows for slightly weaker regularity of the metric.
\end{remark}

%\begin{proof}
%	We fix a Euclidean structure on $\Omega$, with an associated $d$-dimensional Lenesgue measure and $(d-1)$-dimensional Hausdorff measure on $\bd K$. We have to verify that $\Phi$ and $\sigma$ satisfy the assumptions of Theorem \ref{thm:CLT}. Indeed, by definition $\sigma$, as well as the $(d-1)$-dimensional Hausdorff measure on $\bd K$, are given as integrals of continuous volume densities. Since the space of densities on $T_x\bd K$ is one-dimensional for all $x\in\bd K$, the two volume densities differ by multiplication by a positive continuous function. The same applies to $\Phi$ and the Lebesgue measure on $K$. Therefore, we can apply Theorem \ref{thm:CLT} and deduce asymptotic normality of  $\Phi(K_F(n))$.
%\end{proof}

\begin{example}[Hilbert geometry]
	The best known example of a projective Finsler metric is the Hilbert metric inside an open and convex domain $\Omega \subset \RR^d$. The Hilbert--Finsler norm is defined by
	\begin{equation*}
	H_\Omega(x, v) = \frac{1}{2}\left[\frac{1}{t_+(x,v)}   + \frac{1}{t_-(x,v)}\right],
	\end{equation*}
	where $t_\pm(x,v)$ are defined by (see Figure \ref{fig:Hilbert_Finsler})
	\begin{equation}\label{eq:t_pm}
	t_\pm(x,v) = \sup \{t > 0 \,:\, x \pm t v \in \Omega  \}.
	\end{equation}
	
% \begin{columns}[t]
% 	
% 		\begin{column}{0.5\linewidth}
% 		
% 		blab bla
% 	
% 		\end{column}
% 	
% \end{columns}

	\begin{figure}[t]
		\centering
		\begin{tikzpicture}
            \node at (0,0) {\includegraphics[width=6.3cm]{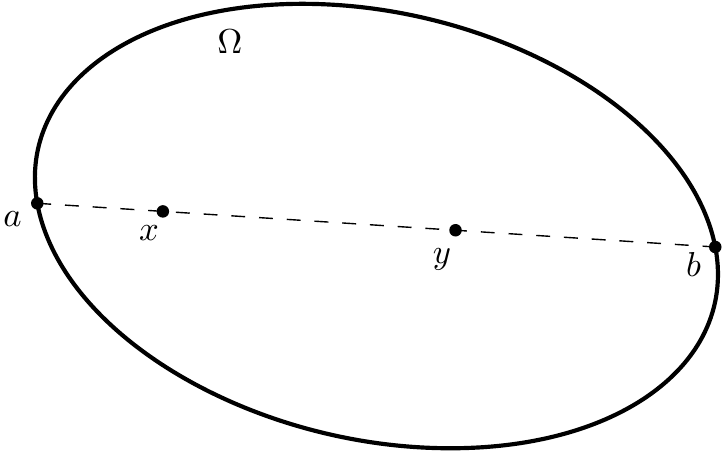}};
            \node at (-2.5,-1.5) {(a)};
            \node at (8,0) {\includegraphics[width=6cm]{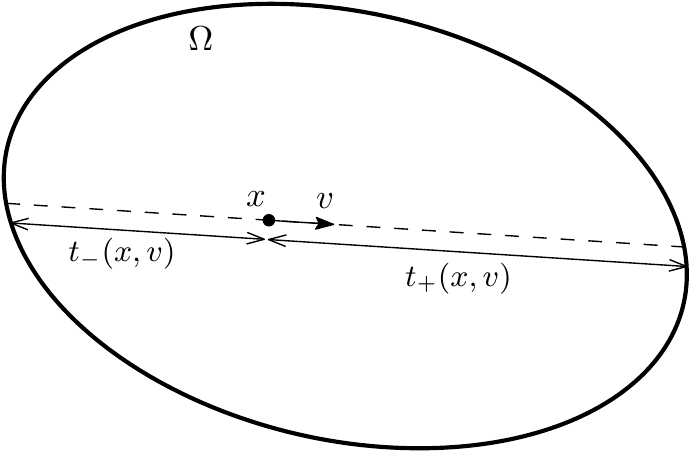}};
            \node at (5.5,-1.5) {(b)};
        \end{tikzpicture}
		\caption{(a) The Hilbert distance in a convex domain $\Omega$. (b) The Finsler norm of a Hilbert (or Funk) geometry in $\Omega$.}\label{fig:Hilbert_Finsler}
	\end{figure}
%\CT{Merge figures and place them on top.}

	\noindent The induced distance function on $\Omega$ is given by 
	\begin{equation*}
	\rho_{\Omega}(x,y) = \frac{1}{2} \log \left( \frac{\|a-x\|}{\|a-y\|} \, \frac{\|b-x\|}{\|b-y\|}\right),
	\end{equation*}
	where $a$ and $b$ are the intersection points of the line passing through $x$ and $y$ with $\bd \Omega$, arranged so that $(a, x, y, b)$ lie in that order on the line (see Figure \ref{fig:Hilbert_Finsler}).
% 	\begin{figure}[h]
% 		\centering
% 		
% 		\caption{The Hilbert distance in a convex domain.}\label{fig:Hilbert_distance}
% 	\end{figure}
 We refer the reader to \cite{PapTroy_Book} for more details on Hilbert geometries.  
 %Assuming that $\Omega$ is a bounded convex domain of class $C^3$, whose boundary has everywhere positive Gaussian curvature, the Finsler norm $H_\Omega$ is $C^3$ and strongly convex.
 From Theorem \ref{thm:Finsler} we deduce that, for any two fixed definitions of volume on $d$- and $(d-1)$-dimensional normed spaces, the Finsler volume of the convex hull of independent random points on the boundary of a convex body in a Hilbert geometry obeys a central limit theorem. 
 
 Let us remark that in order to apply Theorem \ref{thm:Finsler}, we need to assume that the Hilbert--Finsler norm $H_\Omega$ is $C^3$-smooth and strongly convex (which is the case if $\Omega$ is a bounded convex domain of class $C^3$, whose boundary has everywhere positive Gauss--Kronecker curvature.) However, in case of Hilbert geometries we may in fact relax these regularity assumptions, which were made in order to ensure the uniqueness of geodesics. For Hilbert geometries, it is known that this holds if $\Omega$ is a strictly convex domain (see \cite[Corollary 12.7]{PapTroy_From}). Thus is $\Omega$ if a strictly convex domain of class $C^1$, the asymptotic normality of random inscribed polytopes still holds.
\end{example}

\begin{example}[Funk geometry]
%Recall that a \emph{non-reversible} Finsler norm on $\Omega$ is a function $F: T\Omega \to \RR$ such that for any $x \in \Omega$, $F(x, \cdot)$ is a non-reversible norm, i.e., a non-negative function satisfying the triangle inequality and non-degeneracy axiom, and homogeneous with respect to multiplication by \emph{positive} scalars. The \emph{Funk geometry} in a convex domain $\Omega \subset \RR^d$ is the \emph{non-reversible} Finsler norm
%\[
%F_\Omega(x,v) = \frac{1}{t_+(x,v)},
%\]
%where $t_+(x,v)$ is defined as in \eqref{eq:t_pm} (see also Figure \ref{fig:Hilbert_Finsler}). 
%
%
%We refer again to \cite{PapTroy_Book} for more details on Funk geometries. Assume that $\Omega$ is $C^1$ and strictly convex, then Funk metric is uniquely geodesic (see \cite[Corollary 7.8]{PapTroy_From}). Then, fixing either the Buseman or the Holmes-Thompson volume definitions for the $d$- and $(d-1)$- dimensional volumes measurements, we have that the Finsler volume of the convex hull of independent points on the boundary of a convex body in a Funk geometry obeys a central limit theorem. 	

The discussion of definition of volume applies to Finsler norms which are \emph{reversible}, i.e., satisfy $F(x,-v)=F(x,v)$ for all $x \in \Omega$ and $v \in T_x \Omega$. However, for non-reversible Finsler norms, one may still define the Buseman and Holmes--Thompson volume densities. A famous example of a \emph{projective} non-reversible Finsler norm is the \emph{Funk geometry} in a convex domain $\Omega \subset \RR^d$. This is the Finsler norm
\[
    F_\Omega(x,v) = \frac{1}{t_+(x,v)},
\]
where $t_+(x,v)$ is defined as in \eqref{eq:t_pm}, see also Figure \ref{fig:Hilbert_Finsler}. We refer again to \cite{PapTroy_Book} for more details on Funk geometries. Assume that $\Omega$ is $C^1$-smooth and strictly convex, then the Funk metric is uniquely geodesic (see \cite[Corollary 7.8]{PapTroy_From}). Then, fixing either the Buseman or the Holmes--Thompson volume definitions for the $d$- and $(d-1)$-dimensional volume measurements, we have that the Finsler volume of the convex hull of independent random points on the boundary of a convex body in a Funk geometry obeys a central limit theorem. 
\end{example}

\subsection{Dual Brunn--Minkowski theory}\label{sec:dual-BM}

The dual Brunn--Minkowski theory, introduced by Lutwak \cite{Lut75,Lut79}, is a variant of classical Brunn--Minkowski theory, which has become a central piece of modern convex geometry, see e.g.\ \cite{AHH:2018, BHK:2019, BHP:2018, Gardner:2007, GJV:2003, HLYZ:2016, LYZ:2018}. Its starting point is the replacement of Minkowski sum by the so-called \emph{radial sum} of convex bodies, or more generally, star bodies. Dual mixed volumes and related concepts are then derived analogously to classical mixed volumes. While not dual to the classical theory in a precise sense, many results and constructions of the dual theory mirror those of the classical one (see e.g.\ \cite{Sch} for details about dual Brunn-Minkowski theory as well as the references cited therein). Here we focus on the \emph{dual volumes}, which may be derived from dual mixed volumes, or defined directly by dualizing the Kubota formula. Namely, the $j$-th dual volume of a star body $A\subset\RR^d$ is, up to a constant, the average volume of the intersection of $A$ with a $j$-dimensional linear subspace, chosen according to the Haar probability measure on the Grassmannian of $j$-dimensional linear subspaces of $\RR^d$.  In \cite{Lut75b} Lutwak proved the following formula for the $j$-th dual volume of a convex body $A \subset \RR^d$ containing the origin in terms of its \emph{radial function} $\rho_A:S^{d-1}\to (0,+\infty)$, defined by $\rho_A(u) = \max\{r>0: ru\in A\}$:
\begin{equation}\label{eq:dual_volumes}
    \tV_j (A) = \kappa_d \int_{S^{d-1}} \rho_A(u)^j du,
\end{equation}
where we recall that $\kappa_d$ is the volume of the $d$-dimensional Euclidean unit ball.
Moreover, $du$ denotes the infinitesimal element of the normalized surface measure on the unit sphere $S^{d-1}$. Using \eqref{eq:dual_volumes}, Lutwak extended the definition of the dual  volumes ${\tV}_j$ to any $j \in \RR$, and it is this extension which we investigate here. Our next result is a central limit theorem for dual volumes of random inscribed polytopes. However, we note that with positive probability, the random inscribed polytope does not contain the origin. To remedy this, we consider the convex hull of the random polytope with a fixed convex set $T$ containing the origin, which is srictly contained in $K$, where the specific choice of $T$ is irrelevant for our result. Let us emphasize that, for large $n$, the random inscribed polytope contains $T$ with overwhelming probability, in which case this convex hull is simply the polytope itself.

\begin{theorem}\label{thm:dual-volume}
	Let $K \subset \RR^d$ be a convex body of class $C^2_+$ and let $T\subset K$ be another convex body, which is strictly contained in $K$ and contains the origin. Let $\sigma$ be a probablitiy measure on $\bd K$ with positive continuous density. Denote by $K_{\sigma,T}(n) $ the convex hull of the random inscribed polytope $K_\sigma(n)$ and $T$. Then, for any real $j \neq 0$, the dual volume $\tV_j(K_{\sigma,T}(n))$ satisfies a central limit theorem, that is
	\begin{equation}\label{eq:tV_j}
	\frac{ \tV_j(K_{\sigma,T}(n)) - \EE \tV_j(K_{\sigma,T}(n))}{ \sqrt{\Var \tV_j(K_{\sigma,T}(n)) } } \overset{d}{\longrightarrow} Z \qquad \text{as $n\to \infty$},
	\end{equation}
	where $Z$ is a standard Gaussian random variable.
\end{theorem}

\subsection{Random polyhedral sets}
In this section we consider a dual model of a random circumscribing polyhedral set. 
Let $K \subset \RR^d$ be a convex body of class $C^2_+$. Fix a probability measure $\sigma$ on $\bd K$ with a positive and continuous density $\varsigma$ with respect to the $(d-1)$-dimensional Hausdorff measure on $\bd K$. For a point $x \in \bd K$ denote by $H(x)$ the unique supporting affine hyperplane to $K$ at $x$, and by $H^-(x)$ the closed half-space determined by $H(x)$ containing $K$. We define a (weighted) random polyhedral set as follows: Let $X_1, X_2,\ldots$ be a sequence of independent random points on $\bd K$ distributed according to $\sigma$, and define
$$
P_\sigma(n) = \bigcap_{i=1}^n H^-(X_i)
$$
for $n\geq d+1$. We denote by $W(L)$ the mean width of a convex body $L\subset\RR^d$, that is,
$$
W(L) = \int_{S^{d-1}} w(L,u) \,du,
$$
where, as above, the integration is with respect to the normalized spherical Lebesgue measure and $w(L,u)$ is the width of $L$ in direction $u$, i.e., $w(L,u)=h_L(u)+h_L(-u)$ with $h_L(y)=\max\{\langle x,y\rangle:x\in L\}$, $y\in\RR^d$ being the support function of $L$. We show a central limit theorem for the mean width of the random polyhedral set $P_\sigma(n)$. However, as this set is unbounded with positive probability, we will consider its intersection with a fixed convex window $L$ which strictly contains $K$. A common choice for $L$ in the literature is the parallel body $K_1:=\{x \in \RR^d \,:\, {\rm dist} (x, K) \leq 1 \}$, but the result does not depend on the choice of $L$.

\begin{theorem}\label{thm:width}
   % Let $K\subset \RR^d$ be a convex body of class $C_+^2$ and let $L\supset K$ be a convex body that strictly contains $K$ in its interior. Let $\sigma$ be a probability measure on $\bd K$ with positive continuous density. For a sequence $X_1,X_2,\dotsc$ of independent random points on $\bd K$ distributed according to $\sigma$ we consider the polyhedral set $P_\sigma(n)$ that is the intersection of the supporting half-spaces $H^-(X_i)$ of $K$ at $X$. 
    Under the above assumptions, the mean width of $P_\sigma(n) \cap L$ satisfies a central limit theorem, that is,
	\begin{equation*}
	\frac{W(P_\sigma(n) \cap L) - \EE W(P_\sigma(n) \cap L)}{\sqrt{\Var W(P_\sigma(n) \cap L)} } \overset{d}{\longrightarrow} Z \qquad \text{as $n\to\infty$},
	\end{equation*}
	where $Z$ is a standard Gaussian random variable.
\end{theorem}

In this context we would like to mention that the expected mean width of random polyhedral sets $P_\sigma(n) \cap L$ has been analysed in detail in \cite{BFHWidth,BorReit} under different smoothness assumptions on the body $K$. Theorem \ref{thm:width} adds a central limit theorem to this line of research.

\subsection{Weighted random inscribed polytopes}\label{subsect:weighted}

Let us finally turn to the main result of this paper, which we use to derive Theorems \ref{cor:Riemann}, \ref{thm:Finsler} and \ref{thm:dual-volume} as special cases as we shall explain in Section \ref{sec:proof_others}. It is the counterpart for inscribed random polytopes of \cite[Theorem 2.1]{ThaBes}, which holds for random convex hulls with points chosen inside a convex body. At the same time it generalizes the main result in \cite{Tha} for the volume of random convex hulls of uniformly distributed random points on the boundary of a convex body of class $C_+^2$ to weighted volumes and to random points chosen according to a density.

To formally describe the set-up, fix a space dimension $d\geq 2$ and let $K \subset \RR^d$ be a convex body whose boundary $\bd K$ is a $C^2$-smooth submanifold of $\RR^d$ with everywhere positive Gauss--Kronecker curvature. We fix a probability measure $\sigma$ on $\bd K$ with a continuous and positive density $\varsigma > 0$ with respect to the $(d-1)$-dimensional Hausdorff measure on $\bd K$. Additionally,  we let $\Phi$ be a measure on $K$ with a positive density $\phi > 0$ with respect to the Lebesgue measure on $K$, such that $\phi$ is continuous on a (relative) neighbourhood of $ \bd K$ in $K$.

%\DR{Changed the assumptions on $\Phi$.}

This puts us into the position to define what we mean by a \textit{weighted random inscribed polytope} in $K$. We choose a sequence $X_1,X_2,\ldots$ of independent random points on $\bd K$ according to the probability measure $\sigma$, and for $n\geq d+1$ set $K_\sigma(n) := [X_1, \ldots, X_n]$ to be the convex hull of $X_1,\ldots,X_n$. We will prove that the $\Phi$-measure of $K_\sigma(n)$ satisfies a central limit theorem, as $n\to\infty$.

\begin{theorem}\label{thm:CLT}
	Under the above assumptions one has
	\[
	\frac{\Phi(K_\sigma(n)) - \EE \Phi(K_\sigma(n))}{ \sqrt{\Var \Phi(K_\sigma(n)) }} \overset{d}{\longrightarrow} Z\qquad \text{as $n\to \infty$},
	\]
	where $Z$ is a standard Gaussian random variable.
\end{theorem}

\begin{remark}
In our proof we establish the following quantitative version of Theorem \ref{thm:CLT}:
$$
\sup_{t\in\RR}\Big|\PP\Big(\frac{\Phi(K_\sigma(n)) - \EE \Phi(K_\sigma(n))}{ \sqrt{\Var \Phi(K_\sigma(n)) }}\leq t\Big)-\PP(Z\leq t)\Big| \leq C\, n^{-\frac{1}{2}}\,  (\log n)^{2\frac{d+1}{d-1}+1 }
$$
for some constant $C=C(K,\varsigma,\phi)>0$ only depending on $K$, $\varsigma$ and $\phi$. Clearly, taking $n\to\infty$ this yields the distributional convergence stated in Theorem \ref{thm:CLT}. In the same spirit it is possible to upgrade Theorem \ref{cor:Riemann}, Theorem \ref{thm:Finsler}, Theorem \ref{thm:dual-volume} and Theorem \ref{thm:width} as well.
\end{remark}

\section{Preliminaries}

In this paper we denote absolute constants by $c,C,\ldots$ and whenever a constant depends on additional parameters $a,b,\ldots$, say, we indicate this by writing $c=(a,b,\ldots),C=C(a,b,\ldots)$ etc. Our convention is that constants may depend on the convex body $K$ and the measures $\Phi$ and $\sigma$, but never on the number of points $n$.

\subsection{Geometric tools}

Throughout this section we keep the assumptions from Section \ref{subsect:weighted}, namely, $\Omega \subset \RR^d$ is a convex domain, and $K \subset \Omega$ is a convex body of class $C^2_+$, $\Phi$ and $\sigma$ are measures on $K$ and $\bd K$, respectively, with  positive densities $\phi$ and $\varsigma$, with respect to the Lebesgue measure and $(d-1)$-dimensional Hausdorff measure, respectively, such that $\varsigma$ is continuous and $\phi$ is continuous in a neighbourhood of $\bd K$.

For a hyperplane $H \subset \RR^d$ we use the notation $H^\pm$ for the two closed half-spaces bounded by $H$.
For a parameter $t \in (0,1)$, the \textit{$\sigma$-surface body} (or \textit{weighted surface body}) of $K$ with parameter $t$ is defined by
\begin{equation*}
K_\sigma^t = \bigcap \{H^+ \,:\, H \subset \RR^d \text{ a hyperplane}, \sigma(\bd K \cap H^-) \leq t  \}.
\end{equation*}
Note that we get back the classical surface body $K^t$ from \cite{SchWer04} if we choose for $\sigma$
the normalized $(d-1)$-dimensional Hausdorff measure on $\bd K$.

%The following property of the $\sigma$-surface body will be crucial in the following:
%
%\begin{proposition}\label{prop:sandwich}
%	There exists constants $0 < m \leq M < +\infty$, depending on $K$ and $\sigma$, such that for all $t \in (0,1)$ one has 
%	\begin{equation*}
%	K^{mt} \subset K_\sigma^t \subset K^{Mt}.
%	\end{equation*} 
%\end{proposition} 
%
%\begin{proof}
%	Denote by $\cH^{d-1}$ the normalized $(d-1)$-dimensional Hausdorff measure on $\bd K$. By assumption, the density $\varsigma$ of $\sigma$ (with respect to $\cH^{d-1}$) is bounded from above and below, say 
%	\begin{equation}\label{eq:varsigma_bound}
%	0 < \frac{1}{M}\leq \varsigma \leq \frac{1}{m} < +\infty.
%	\end{equation} Let $H$ be any hyperplane which appears in the definition of $K_\sigma^t$, that is $\sigma (H^+\cap \bd K) \leq t$. Then from \eqref{eq:varsigma_bound} we see that $\cH^{d-1}(H^+\cap \bd K) \leq Mt$. This obviously implies that
%	\begin{equation*}
%	K_\sigma^t \subset K^{Mt}.
%	\end{equation*}
%	By a symmetric argument, we get the second inclusion, which completes the proof.
%\end{proof}

For a point $z \in \bd K$ and $t > 0$, define the \emph{visibility region} of $z$ (with respect to the measure $\sigma$) as all points in $K \setminus K_\sigma^t$ visible from $z$ around the ‘obstacle’ $K_\sigma^t$, that is 
\begin{equation*}
\Vis_\sigma(z, t) = \{ y \in K \setminus K_\sigma^t \,:\, [z,y] \cap K_\sigma^t = \emptyset  \}.
\end{equation*}
\begin{figure}[h]
	\centering
	\begin{tikzpicture}[scale=1.2]
	\clip (-3.5,-0.5) rectangle (3.5,2.5);
	\path[use as bounding box] (-3.5,-0.5) rectangle (3.3,2.5);
	\begin{scope}[rotate=-20, xscale=3, yscale=2, rotate=20] % affine transformation to make circles look fancy...
	
	%controls $K_\delta$
	\def\r{0.89}
	
	% K circle
	\draw (0,0) circle (1);
	% K_\delta circle
	\draw[dashed] (0,0) circle (\r);
	
	%z
	\coordinate (O) at (0,1);
	
	%intersection with K_\delta
	\coordinate (A1) at ({\r*sqrt(1-\r*\r)}, {\r*\r});
	\coordinate (B1) at ({2*\r*sqrt(1-\r*\r)},{2*\r*\r-1});
	
	%intersection with K
	\coordinate (A2) at ({-\r*sqrt(1-\r*\r)}, {\r*\r});
	\coordinate (B2) at ({-2*\r*sqrt(1-\r*\r)},{2*\r*\r-1});
	
	%Delta(z,delta)
	\fill[color=black, opacity=0.2] (A2)
	arc({180-atan(\r/sqrt(1-\r*\r))}:{atan(\r/sqrt(1-\r*\r))}:\r) -- (B1)
	arc({atan((2*\r*\r-1)/(2*\r*sqrt(1-\r*\r))}:{180-atan((2*\r*\r-1)/(2*\r*sqrt(1-\r*\r))}:1) -- cycle;
	\draw[thick] (A1) -- (B1) arc({atan((2*\r*\r-1)/(2*\r*sqrt(1-\r*\r))}:{180-atan((2*\r*\r-1)/(2*\r*sqrt(1-\r*\r))}:1) -- (A2);
	\draw[thick, white]  (A2) arc({180-atan(\r/sqrt(1-\r*\r))}:{atan(\r/sqrt(1-\r*\r))}:\r);
	%\draw[thick,dashed] (A2) arc({180-atan(\r/sqrt(1-\r*\r))}:{atan(\r/sqrt(1-\r*\r))}:\r);
	\draw[thick] (A2) arc({180-atan(\r/sqrt(1-\r*\r))}:{atan(\r/sqrt(1-\r*\r))}:\r);
	\draw[dotted] (A1) -- (O) -- (A2);
	
	%%%% Circle marker should be a circle after the affine transformation...
	\begin{scope}[yshift=1cm]
	\fill[black] (0,0) [rotate=-20, xscale = 1/3, yscale = 1/2, rotate=20] circle(0.05) node[above right] {$z$};
	\end{scope}

	%%% normal vector and cap
	%\coordinate (O1) at ([rotate=-40, xscale=1/3, yscale = 1/2, rotate=40] 0,0.8);
	%\coordinate (O2) at ([rotate=0, xscale=1/3, yscale = 1/2, rotate=0] 4,0  );
	%\draw[->,>=stealth] (O) --++ (O1) node[left] {$n_z$};
	
	%\draw[dotted] (O) --++ ([rotate=180, scale=2.2] O1);
	%\begin{scope}
	%\clip circle (1);
	%\fill[opacity=0.1] (O) ++ ([rotate=180, scale=2.2] O1) ++ (O2)
	%--++ ([scale=-2]O2)  --++ ([scale=4.4] O1) --++ ([scale=2] O2) -- cycle;
	%\draw[dotted] (O) ++ ([rotate=180, scale=2.2] O1) ++ (O2) --++ ([scale=-2]O2);
	%\end{scope}
	
	\end{scope}
	
	%\node[rotate=-5] at (0.7,0.) {$C\big(z,c\delta^{\frac{2}{d+1}}\big)$};
	\node at (-3.2,0.) {$K$};
	\node at (-2,0.) {$K_\sigma^t$};
	\node at (2,1.8) {$\Vis_\sigma(z,t)$};
	\end{tikzpicture}
	
	\caption{The visibility region $\Vis_\sigma(z,t)$ of a point $z\in\bd K$.}
	\label{fig:vis_region}
\end{figure}
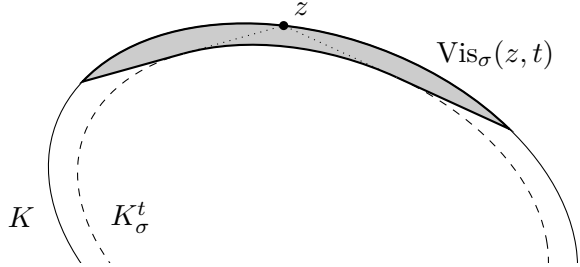
As above, when $\sigma$ is the $(d-1)$-dimensional Hausdorff measure we denote the visibility region simply by $\Vis(z, t)$. We will require the following estimates on visibility regions:
\begin{lemma}
	Let $K \subset \RR^d$ be a convex body of class $C^2_+$. Then there is a constant $C=C(K, \sigma, \Phi)$ such that for all sufficiently small $t>0$ one has
	\begin{equation}\label{eq:bound-Phi-Vis}
	\sup_{z \in \bd K} \Phi(\Vis_\sigma(z, t)) \leq C t^{\frac{d+1}{d-1}}
	\end{equation}
	and
	\begin{equation}\label{eq:bound-Vis-disjoint}
	\sup_{z\in \bd K} \sigma (\left\{ y \in \bd K \,:\, \Vis_\sigma(z, t) \cap \Vis_\sigma(y,t) \neq \emptyset  \right\}) \leq C t.
	\end{equation}
\end{lemma}

The proofs of \eqref{eq:bound-Phi-Vis} and \eqref{eq:bound-Vis-disjoint} for the unweighted  case (i.e., when $\Phi$ is the Lebesgue measure and $\sigma$ is the $(d-1)$-dimensional Hausdorff measure) can be extracted from existing literature (see \cite[Lemma 6.3]{Vu} and \cite[Lemma 6.2]{RVW}), and the general case follows by a `sandwiching' argument similar to \cite[Lemma 5.2]{BLW}. For transparency, we sketch a direct argument below.

\begin{proof}
The result follows from elementary properties of caps. By definition, a cap in $K$ is a subset of the form $K \cap H^+$, where $H$ is an affine hyperplane. Any cap $C^K$ contains a unique point $z \in \bd K$ of maximal distance from $H$, which we call the \emph{center} of $C^K$. When $\sigma(C^K \cap \bd K)=t$, we call $C^K$ a $t$-cap. The important (and trivial) observation here is that $\Vis_{\sigma}(z,t)$ is precisely the union of all $t$-caps containing $z$. 

The proof requires the following estimates on caps: there exists positive constants $M, t_0, \rho_0$, depending only on $K$, $\sigma$ and $\Phi$, for which the following holds.
\begin{itemize}%[(i)]
	\item Any $t$-cap with $t \leq t_0$ has diameter $\leq M t^{\frac{1}{d-1}}$ and $\Phi$-measure $\leq M t^{\frac{d+1}{d-1}}$.
	\item Any subset $Y \subset \bd K$ with diameter $\rho \leq \rho_0$ is contained in the $t$-cap centered at any point $y \in Y$, where $t=M \rho^{d-1}$.
\end{itemize}
These facts can be proven by a simple direct computation, using the fact that our assumptions on $K$ imply uniform upper and lower (away from zero) bounds on the principle curvatures. Assuming this, it easily follows that for any $z\in \bd K$, $\Vis_\sigma(z,t)$ is contained in the $(Mt)$-cap centered at $z$. Using the bound on $\Phi$-measure in the first item, this fact implies \eqref{eq:bound-Phi-Vis}. Moreover, this fact also implies uniform bounds on the diameter of the sets $\left\{ y \in \bd K \,:\, \Vis_\sigma(z, t) \cap \Vis_\sigma(y,t) \neq \emptyset  \right\}$, which by the second item implies \eqref{eq:bound-Vis-disjoint}.
\end{proof}

%\begin{proof}
%	We first recall that in the case where $\Phi$ is the Lebesgue measure and $\sigma$ is the $(d-1)$-dimensional Hausdorff measure, then
%	\eqref{eq:bound-Phi-Vis} can be found in \cite[Eq.\ (1)]{Tha} and \eqref{eq:bound-Vis-disjoint} is equal to \cite[Eq.\ (2)]{Tha}.
%	To reduce to this case, we note that from Proposition \ref{prop:sandwich} it follows, for the same constant $0 < m \leq M < +\infty$, that one has
%	\begin{equation}\label{eq:Viz_sandwich}
%	\Vis(z, Mt) \subset \Vis_{\sigma}(z, t) \subset \Vis(z, mt).
%	\end{equation}
%	This, together with the fact that $\Phi$ and $\sigma$ have  bounded densities with respect to the Lebesgue measure, and the $(d-1)$-dimensional Hausdorff measure, respectively, easily imply \eqref{eq:bound-Phi-Vis} and  \eqref{eq:bound-Vis-disjoint} in the general case.
%\end{proof}

Finally, we will make extensive use of the fact that $K_\sigma(n)$ contains the $\sigma$-surface body with overwhelming probability. More precisely, we require the following result from \cite[Lemma 4.2]{RVW}.
\begin{lemma}\label{lem-RVW-surface-body}
Let $K \subset \RR^d$ be a convex body of class $C^2_+$, and let $\sigma$ be a probability measure on $\bd K$ with positive and continuous density with respect to the $(d-1)$-dimensional Hausdorff measure. Then for any $\alpha > 0$ there exists $c =c(\alpha) > 0$ such that for $n$ sufficiently large one has, denoting $\tau = c \,\frac{\log n}{n}$,
\begin{equation*}%\label{eq:Vu-surface-body}
\PP \bigl( K_\sigma^\tau \not\subset K_\sigma(n) \bigr) \leq n^{-\alpha}.
\end{equation*}
\end{lemma}

\subsection{A normal approximation bound}

The purpose of this section is to rephrase a very general normal approximation bound for non-linear functionals of independent and identically distributed random variables from \cite{Chatterjee,ChiShaoZhang,LachPecc}. We present it in the framework of general Polish spaces $E$ with a probability measure $\mu$, although in our application in the proof of Theorem \ref{thm:CLT} $E$ will be the boundary of a smooth convex body in $\RR^d$ and $\mu$ the probability measure $\sigma$ on $\bd K$. For $n\in\NN$, let $f:\bigcup_{k=1}^nE^k\to\RR$ be a symmetric and measurable function.
By this we mean that $f$ is a symmetric function acting on point configurations in $E$ of at most $n$ points. If $x=(x_1,\ldots,x_n)\in E^n$ and $i\in\{1,\ldots,n\}$, we introduce the notation
\begin{equation*}
x_{\neg i}:= (x_1,\dotsc,x_{i-1}, x_{i+1},\dotsc,x_n)\in E^{n-1}.
\end{equation*}
Similarly, for two indices $i,j\in\{1,\ldots,n\}$ with $i<j$ we denote by $x_{\neg i,j}\in E^{n-2}$ the $(n-2)$-tuple arising from $x$ by removing
both $x_i$ and $x_j$. We are now in the position to define the first- and second-order difference operator of $f$ by
\begin{align*}
D_i f(x) &:= f(x)-f(x_{\neg i}),
\intertext{and}
D_{i,j} f(x) &:= D_i(D_j f(x)) = f(x)-f(x_{\neg i})-f(x_{\neg j})+f(x_{\neg i,j}),
\end{align*}
respectively. In other words, $D_if(x)$ measures the effect on the functional $f$ when $x_i$ is removed from $x$, and similar interpretation is valid for $D_{i,j}f(x)$.

Let now $X=(X_1,\ldots,X_n)$ be an $n$-tuple of independent random elements from $E$ with distribution $\mu$, and let $X'$ and $X''$
be independent random copies of $X$ whose coordinates are denoted by $X_i'$ and $X_i''$, $i\in\{1,\ldots,n\}$, respectively.
By a recombination of $\{X,X',X''\}$ we understand a random vector $Z=(Z_1,\ldots,Z_n)$ having the property that $Z_i\in\{X_i,X_i',X''_i\}$ for each $i\in\{1,\ldots,n\}$. This allows us to introduce the following three quantities:
\begin{align*}
B_1(f) & := \sup_{(Y,Y',Z,Z')}\EE\big[\mathbf{1}\{D_{1,2}f(Y)\neq 0\}\mathbf{1}\{D_{1,3}f(Y')\neq 0\}\,(D_2f(Z))^2(D_3f(Z'))^2\big]\,,\\
B_2(f) & := \sup_{(Y,Z,Z')}\EE\big[\mathbf{1}\{D_{1,2}f(Y)\neq 0\}\,(D_1f(Z))^2(D_2f(Z'))^2\big]\,,\\
B_{3}(f) &:= \EE|D_1f(X)|^4\,,
\end{align*}
where in the definition of $B_1$ the supremum is taken over all $4$-tuples of random vectors $Y$, $Y'$, $Z$, and $Z'$,
which are recombinations of $\{X,X',X''\}$, and in the definition of $B_2$ the supremum is taken over all $3$-tuples of random vectors $Y$, $Z$ and $Z'$, which are recombinations of $\{X,X',X''\}$.

To measure the distance between (the laws of) two random variables $W$ and $V$ we use the \textit{Kolmogorov distance}. We recall that the Kolmogorov distance between $W$ and $V$ is given by
\begin{equation}\label{eq:DefKolmogorov}
d_{\Kol}(W,V) := \sup_{x\in\RR}\big|\PP(V\leq x)-\PP(W\leq x)\big|\,,
\end{equation}
and note that convergence of the Kolmogorov distance implies convergence in distribution.

We are now prepared to rephrase the following normal approximation bound from \cite{ChiShaoZhang}, which is essentially based on the previous works \cite{Chatterjee, LachPecc}.

\begin{lemma}\label{lem:CLTLachiezeReyPeccati}
	Fix $n\in\mathbb{N}$. Let $X_1,\dotsc,X_n$ be independent random elements taking values in a Polish space $E$ and are distributed according to a probability measure $\mu$, and let
	$f:\bigcup_{k=1}^n E^k\to\RR$ be a symmetric and measurable function.
	Define $W(n):=f(X_1,\ldots,X_n)$ and assume that $\EE\, W(n)=0$ and $\EE\, W(n)^2=1$. Then there exists an absolute constant $c>0$ such that
	\begin{equation}\label{eq:BoundWassersteinGeneral}
	\begin{split}
	d_{\Kol}\left(W(n),Z\right)
	&\leq c\sqrt{n}\left(n\sqrt{B_1(f)}+\sqrt{nB_2(f)}+\sqrt{B_{3}(f)}\right),
	\end{split}    
	\end{equation}
	where $Z$ is a standard Gaussian random variable.
\end{lemma}

\section{Proof of Theorem \ref{thm:CLT}}\label{sec:proof_main_thm}

This section is devoted to the proof of our main result about weighted random inscribed polytopes. The proof uses Lemma~\ref{lem:CLTLachiezeReyPeccati}. To obtain the required bound on the right hand side of \eqref{eq:BoundWassersteinGeneral} we need to combine an upper bound on the difference operators with a lower bound on the variance. As the two are independent, we treat them separately, the latter in Section~\ref{subsec:variance} and the former in Section~\ref{subsec:proof_main}

\subsection{A lower bound for the variance}\label{subsec:variance}

Richardson, Vu and Wu \cite[Theorem 1.1]{RVW} established a lower bound for the variance of the volume of the random inscribed polytope $K_\sigma(n)$ by adapting the proof of Reitzner \cite[Theorem 3]{Reitzner:2005}. With some further adaptions to their arguments we will show that the the following more general theorem holds as well.

\begin{theorem}\label{thm:var-bound}
	Let $K\subset \RR^d$ be a convex body of class $C_+^2$ and fix a probability measure $\sigma$ on $\bd K$ with continuous density $\varsigma>0$ with respect to the $(d-1)$-dimensional Hausdorff measure. Then set $K_\sigma(n)$ as the random inscribed polytope generated as the convex hull of $n$ independent random points distributed with respect to $\sigma$.
	Furthermore, let $\Phi$ be a measure on $K$ with continuous density $\phi>0$ with respect to the Lebesgue measure on $K$. Then there exist constants $c=c(K,\varsigma,\phi)>0$ and $N=N(K,\varsigma,\phi)\in\mathbb{N}$ such that for all $n\geq N$ we have that
	\begin{equation*}
	\Var \Phi(K_\sigma(n)) \geq cn^{-\frac{d+3}{d-1}}.
	\end{equation*}
\end{theorem}

%We may follow the proof of \cite[Theorem 1.1]{RVW} almost verbatim and choose to reproduce here only the parts that have to be adapted when changing from the Euclidean volume $\Vol_d$ to a weighted volume measure $\Phi$.

One of the key constructions is to approximate $\bd K$ around a fixed point $x\in \bd K$ by an elliptic paraboloid $Q_x$. If we choose coordinates such that $x$ is at the origin and such that $\RR^{d-1}$ is the tangent hyperplane to $\bd K$ at $x$ where the outer unit normal $n_K(x)$ of $\bd K$ at $x$ is $-e_d$. Then
\begin{equation*}
Q_x := \{z\in\RR^d : \kappa_1(x) z_1^2 + \dotsc +\kappa_{d-1}(x) z_{d-1}^2 \leq 2 z_d\},
\end{equation*}
where $\kappa_1(x),\dotsc,\kappa_{d-1}(x)$ are the principal curvatures of $\bd K$ at $x$.
We may map the standard elliptic paraboloid $E=\{z\in\RR^d: z_1^2+\dotsc+z_{d-1}^2 \leq z_d\}$ to $Q_x$ via a linear map, i.e., if we set
\begin{equation}\label{eq:MapA}
A_x := \mathrm{diag}\left(\sqrt{\frac{2h}{\kappa_1(x)}},\dotsc,\sqrt{\frac{2h}{\kappa_{d-1}(x)}},h\right),
\end{equation}
then $Q_x = A_xE$. Here, the dependence on $h>0$ is chosen in such a way that the cap $C^{E}(0,1):=\{z\in E : z_d\leq 1\}$ of height $1$ is mapped to
\begin{equation*}
C^{Q_x}(x,h) := \{z\in Q_x: z_d\leq h\} = A_xC^{E}(0,1).
\end{equation*}
Note also that
\begin{equation*}
\det A_x = 2^{\frac{d-1}{2}} \kappa(x)^{-\frac{1}{2}} h^{\frac{d+1}{2}},
\end{equation*}
where $\kappa(x) := \prod_{i=1}^{d-1} \kappa_i(x)$ is the Gauss--Kronecker curvature of $\bd K$ at $x$.
Since $K$ is of class $C_+^2$ there are $h_0>0$ and $c_0>1$ such that for all $h\in (0,h_0)$ we have that
\begin{equation}\label{eqn:det_A}
\frac{1}{c_0} h^{\frac{d+1}{2}} \leq \left|\det A_x\right| \leq c_0 h^{\frac{d+1}{2}}.
\end{equation}
Since the (Lebesgue) density function $\phi$  of $\Phi$ is positive and continuous near $\bd K$, we also find $c_1>1$ such that for all $h$ small enough we have that
\begin{equation*}
\frac{1}{c_1} \Vol_d(C^{Q_x}(x,h)) (\phi(x)+o_h(1)) \leq  \Phi(C^{K}(x,h)) \leq c_1  \Vol_d(C^{Q_x}(x,h)) (\phi(x)+o_h(1)),
\end{equation*}
where $o_h(1) \to 0$ as $h\to 0^+$, and $C^K(x,h):= \{z\in K : \langle x-z,n_K(x) \rangle \leq h\}$ is the cap of $K$ of height $h$ with apex in $x$. Here we use the fact that for $h>0$ small enough (independently of $x \in \bd K$), the cap $C^{K}(x,h)$ is contained in the neighbourhood of $\bd K$ where $\phi$ is continuous.
%\DR{I corrected the assumptions on $\phi$, and added an explanation.}

\smallskip
Next, let us repeat the random simplex construction in the standard paraboloid $E$. 
In the following we denote by $H(u,t)$ the affine hyperplane with unit normal $u\in S^{d-1}$ and signed distance $t$ from the origin, i.e., $H(u,t) = \{x\in\RR^d : \langle x, u\rangle = t\}$.
We first consider the simplex $S=[v_0,\dotsc,v_d]$ in the cap $C^{E}(0,1)$, where the vertex $v_0$ is the origin and $[v_1,\dotsc,v_d]$ is a regular simplex inscribed in the $(d-1)$-dimensional ball $\{z\in E: z_d = h_d\}$, where $h_d<\frac{1}{2d^2}$ is chosen small enough so that
\begin{equation*}
    \{\lambda z : \lambda\geq 0,\, z\in S\} \supset (2E)\cap H(e_d,1) = \{(z_1,\dotsc,z_{d-1},1) \in \RR^d: \|(z_1,\dotsc,z_{d-1})\|=\sqrt{2}\}.
\end{equation*}
This condition ensures that the cone spanned by $S$ is ``flat'' enough and will be important later on, see \eqref{eqn:independence-cond}, to ensure a certain independence property, see \eqref{eqn:independence_cond2}.

\smallskip
Now we consider the orthogonal projection $\mathrm{proj}_{\RR^{d-1}}: \RR^d \to \RR^{d-1}$, defined by $\mathrm{proj}_{\RR^{d-1}}(z) = (z_1,\dotsc,z_{d-1})$. We consider a balls $B_i\subset \RR^{d-1}$ of radius $r>0$ around $v_0=0=v_0'$ and $\mathrm{proj}_{\RR^{d-1}}(v_i)$ for $i=1,\dotsc,d$. We further set $B_i' := \bd E\cap \mathrm{proj}_{\RR^{d-1}}^{-1}(B_i)$ for $i=0,\dotsc,d$. We will choose $r>0$ later, but it will be small enough so that for all $w_i\in B_i'$, $i=0,\dotsc,d$, we have that $[w_0,\dotsc,w_d]$ is sufficiently close to $S=[v_0,\dotsc,v_d]$. In particular, we have that
\begin{equation*}
\{\lambda z: \lambda\geq 0,\, z\in [w_0,\dotsc,w_d]\} \supset (2E)\cap H(e_d,1),
\end{equation*}
for all $w_i\in B_i'$, $i=0,\dotsc,d$ (see Figure \ref{fig:simplex}).
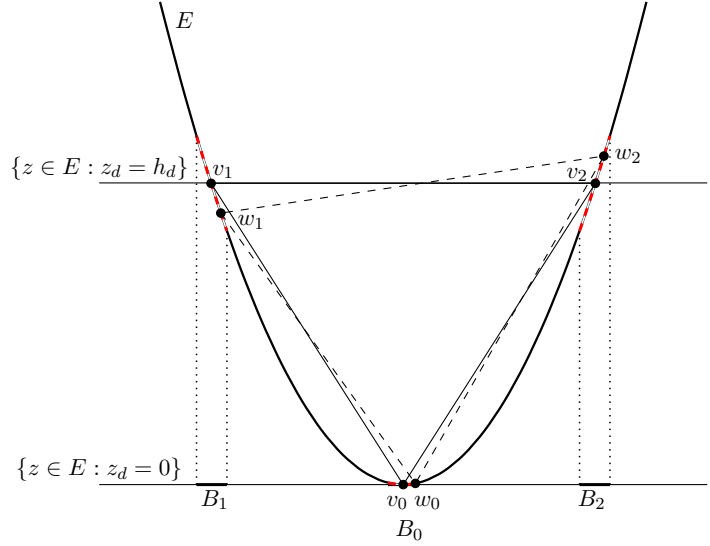
\begin{figure}[t]
	\centering
	\begin{tikzpicture}[scale=0.8, transform shape]
	\draw[black, line width = 0.30mm]   plot[smooth,domain=-4:4] (\x, {0.5*(\x)^2});
	\draw (-5,5) -- (5,5);
	\draw (-5,0) -- (5,0);
	
	\draw[white, line width = 0.20mm]   plot[smooth,domain=-3.4:-2.9] (\x, {0.5*(\x)^2});
	\draw[black, line width = 0.40mm,dashed,red]  plot[smooth,domain=-3.4:-2.9] (\x, {0.5*(\x)^2});
	\draw[black,line width=0.4mm] (-3.4,0) -- (-2.9,0);
	\draw[black,line width=0.4mm] (3.4,0) -- (2.9,0);
	\draw[black,line width=0.4mm] (-0.25,0) -- (0.25,0);
	\draw[black,dotted,line width=0.2mm] (-3.4,0) -- (-3.4,0.5*3.4*3.4);
	\draw[black,dotted,line width=0.2mm] (-2.9,0) -- (-2.9,0.5*2.9*2.9);
	\draw[black,dotted,line width=0.2mm] (3.4,0) -- (3.4,0.5*3.4*3.4);
	\draw[black,dotted,line width=0.2mm] (2.9,0) -- (2.9,0.5*2.9*2.9);
	
	\draw[white, line width = 0.20mm]   plot[smooth,domain=2.9:3.4] (\x, {0.5*(\x)^2});
	\draw[black, line width = 0.40mm,dashed,red]   plot[smooth,domain=2.9:3.4] (\x, {0.5*(\x)^2});
	
	\draw[white, line width = 0.20mm]   plot[smooth,domain=-0.25:0.25] (\x, {0.5*(\x)^2});
	\draw[black, line width = 0.40mm,dashed,red]   plot[smooth,domain=-0.25:0.25] (\x, {0.5*(\x)^2});
	
	\filldraw (-3.16,0.5*3.16*3.16) circle (2pt);
	\filldraw (-3,0.5*3*3) circle (2pt);
	\filldraw (3.16,0.5*3.16*3.16) circle (2pt);
	\filldraw (3.3,0.5*3.3*3.3) circle (2pt);
	\filldraw (0,0) circle (2pt);
	\filldraw (0.2,0.5*0.2*0.2) circle (2pt);
	
	\draw (-3.16,0.5*3.16*3.16) -- (0,0) -- (3.16,0.5*3.16*3.16) -- (-3.16,0.5*3.16*3.16);
	\draw[dashed] (-3,0.5*3*3) -- (0.2,0.5*0.2*0.2) -- (3.3,0.5*3.3*3.3) -- (-3,0.5*3*3);
	
	\node at (-0.1,-0.3) {$v_0$};
	\node at (0.4,-0.3) {$w_0$};
	\node at (-3.16+0.2,0.5*3.16*3.16+0.2) {$v_1$};
	\node at (3.16-0.3,0.5*3.16*3.16+0.15) {$v_2$};
	\node at (-3+0.5,0.5*3*3-0.15) {$w_1$};
	\node at (3.3+0.4,0.5*3.3*3.3) {$w_2$};
	\node at (-3.6,7.7) {$E$};
	\node at (-5,5.25) {$\{z\in E:z_d=h_d\}$};
	\node at (-5,0.25) {$\{z\in E:z_d=0\}$};
	\node at (-3.1,-0.25) {$B_1$};
	\node at (3.1,-0.25) {$B_2$};
	\node at (0.1,-0.75) {$B_0$};
	\end{tikzpicture}
	\caption{Construction of the random simplex $[w_0, \ldots, w_d]$}
	\label{fig:simplex}
\end{figure}

Furthermore, if $W$ is randomly distributed on $B_0'$ with respect to a continuous and positive probability density $\vartheta$, then there exists $c_3>0$ such that
\begin{equation*}
\Var_W(\Vol_d([W,v_1,\dotsc,v_d])) \geq c_3>0,
\end{equation*}
where for some random element $X$ the notation $\Var_X$ (and also $\EE_X$ below) indicates that the variance (or the expectation) is taken with respect to the law of $X$.

\smallskip
Using the linear transformation $A_x$ defined at \eqref{eq:MapA}, we set 
\begin{equation*}
D_i(x) := A_x B_i \subset T_x\bd K\cong \RR^{d-1},
\end{equation*}
where $T_x\bd K$ is the tangent space of $\bd K$ at $x$, which is isometric ($\cong$) to $\RR^{d-1}$. Further, we set
\begin{equation*}
D_i'(x) := \tilde{f}_x(D_i(x)) \subset \bd K \cap C^K(x,h),
\end{equation*}
where $f_x:\RR^{d-1}\to\RR$ locally defines $\bd K$ around $x$ via $\tilde{f}_x(y) = (y,f_x(y)) \in \bd K$. We also stress that $D_i'(x)\subset \bd K$ is \emph{not} the image of $B_i'$ under $A_x$ since $A_xB_i'\subset \bd Q_x$.
Finally, for sufficiently small $h>0$, we find that
\begin{equation*}
\frac{1}{c_4} h^{\frac{d-1}{2}} \leq \sigma(D_i'(x)) \leq c_4 h^{\frac{d-1}{2}},
\end{equation*}
for some constant $c_4>1$ and
\begin{equation}\label{eqn:independence-cond}
\begin{split}
\{\lambda z: \lambda\geq 0, \, z\in [y_0,\dotsc,y_d]\} &\supset (2Q_x)\cap H(n_K(x),\langle x,n_K(x) \rangle -h)\\
&\supset K \cap H(n_K(x), \langle x, n_K(x)\rangle-h),
\end{split}
\end{equation}
for all $y_i\in D'_i(x)$.

\medskip
We are now ready to adapt to our situation the main lemma \cite[Lemma 3.1]{RVW}, that has to be changed in the proof of \cite[Theorem 1.1]{RVW}.

\begin{lemma}\label{lem:var-bound2}
	There exists $r_0>0$ such that for all $r\in(0,r_0)$ there is $h_0=h_0(r)>0$ and $c_5=c_5(r)>1$ such that for all $y_i\in D_i'(x)$, $i=1,\dotsc,d$, and $h\in(0,h_0)$ we have that
	\begin{equation}\label{eqn:local_variance_bound}
	\frac{1}{c_5} h^{d+1} \leq \Var_Y \Phi([Y,y_1,\dotsc,y_d]) \leq c_5 h^{d+1},
	\end{equation}
	where $Y$ is a random point in $D_0'(x)\subset \bd K$ distributed with respect to a continuous density function $\varsigma>0$.
\end{lemma}
\begin{proof}
	To prove \cite[Lemma 3.1]{RVW} one first shows \cite[Claim 8.1]{RVW}, where the first and second moment of the volume are asymptotically bounded. Following the proof of \cite[Claim 8.1]{RVW} and \cite[Claim 8.2]{RVW} we obtain
	\begin{align*}
	\EE_Y\Phi([Y,y_1,\dotsc,y_d]) 
	&= (1+o_{r,h}(1)) \frac{1}{\Vol_{d-1}(B_0)} \int_{B_0} \Phi([\tilde{f}_x(A_x'z),y_1,\dotsc,y_d]) \, \dint z,
	\end{align*}
	where we recall from \eqref{eq:MapA} that $A_x$ is the linear transformation that maps the standard paraboloid $E$ to the approximating paraboloid $Q_x=A_xE$ of $\bd K$ around $x$ and $A_x'$ is the restriction of $A_x$ to $\RR^{d-1}$, i.e., $A_x'=\mathrm{diag}\left(\sqrt{\frac{2h}{\kappa_1(x)}},\dotsc,\sqrt{\frac{2h}{\kappa_{d-1}(x)}}\right)$. Now, since $\phi$ is continuous at $x\in\bd K$ and since $[Y,y_1,\dotsc, y_d]\subset C^K(x,h)$, we find
	\begin{align*}
	\Phi([\tilde{f}_x(A_x'z),y_1,\dotsc,y_d]) 
	&= \Vol_d([\tilde{f}_x(A_x'z), y_1, \dotsc, y_d]) (\phi(x)+o_{h}(1)).
	\end{align*}
	%\DR{Here I didn't add anything, since it's already remarked above that we're using continuity of $\phi$ at $x$.}
	By setting
	\begin{equation*}
	\psi_1(r) := \frac{1}{\Vol_{d-1}(B_0)} \int_{B_0} \Vol_d([\tilde{b}(z),v_1,\dotsc,v_d])\, \mathrm{d}z,
	\end{equation*}
	where $\tilde{b}(z) = (z,\|z\|^2)$ parameterizes $E$, we therefore derive
	\begin{equation*}
	\lim_{h\to 0^+} \frac{\EE_Y \Phi([Y,y_1,\dotsc,y_d])}{\left|\det A_x\right|\psi_1(r)} = (\phi(x)+o_r(1)),
	\end{equation*}
	similar to \cite[Equation (33)]{RVW}. Analogously, by setting
	\begin{equation*}
	\psi_2(r) := \frac{1}{\Vol_{d-1}(B_0)} \int_{B_0} \Vol_d([\tilde{b}(z),v_1,\dotsc,v_d])^2\, \mathrm{d}z,
	\end{equation*}
	we obtain
	\begin{equation*}
	\lim_{h\to 0^+} \frac{\EE_Y \Phi([Y,y_1,\dotsc,y_d])^2}{\left|\det A_x\right|^2\psi_2(r)} = (\phi(x)^2+o_r(1)).
	\end{equation*}
	Hence,
	\begin{align*}
	\lim_{h\to 0^+} \frac{\Var_Y \Phi([Y,y_1,\dotsc,y_d])}{\left|\det A_x\right|^2} 
	&= (\psi_2(r) - \psi_1(r)^2)\,(\phi(x)^2+o_r(1))\\
	&= \big[\!\Var_W\Vol_d([W,v_1,\dotsc,v_d])\big]\, (\phi(x)^2+o_r(1)) >0
	\end{align*}
	for $r>0$ small enough, since $\phi(x)>0$ and $\Var_W\Vol_d([W,v_1,\dotsc,v_d])>0$ for all $r>0$.
	%\FB{This estimate relates to \cite[Eq.\ (35)]{RVW} and is essentially the only part that has to change in the proof of \cite[Theorem 1.1]{RVW}, the rest of this section is just an attempt to make this somehow understandable, since this estimate is buried quite deep in their proof...}
	Thus, there is $c_6>1$ and $h_0>0$ such that for all $h\in(0,h_0)$ we have that
	\begin{equation*}
	\frac{1}{c_6} \left|\det A_x\right|^2 \leq \Var_Y \Phi([Y,y_1,\dotsc,y_d]) \leq c_6\left|\det A_x\right|^2,
	\end{equation*}
	which completes the proof by \eqref{eqn:det_A}.
\end{proof}

\begin{proof}[Proof of Theorem \ref{thm:var-bound}]
	Replacing \cite[Lemma 3.1]{RVW} with Lemma \ref{lem:var-bound2} in the proof of \cite[Theorem 1.1]{RVW} essentially yields the statement. Let us briefly recall the main steps: Choose $n$ points $Y_1,\dotsc,Y_n$ in $\bd K$ at random according to $\varsigma$. Furthermore, choose $n$ points $x_1,\dotsc,x_n\in \bd K$ and corresponding disjoint caps $C^K(x_j,h_n)$, $j=1,\dotsc,n$, according to the economic cap covering, see \cite[Lemma 6.6]{RVW}, and in each cap $C^K(x_j,h_n)$ define the sets $D_i'(x_j)$, $i=0,\dotsc,d$ as constructed before.
	Here,
	\begin{equation*}
	\frac{1}{c_7} n^{-\frac{2}{d-1}} \leq h_n \leq c_7 n^{-\frac{2}{d-1}},
	\end{equation*}
	and 
	\begin{equation*}
	\frac{1}{c_8 n} \leq \sigma(C^K(x_j,h_n)\cap \bd K) \leq \frac{c_8}{n},
	\end{equation*}
	for some constants $c_7,c_8>1$ and $n$ large enough.
	
	\smallskip
	Now let $A_j$, $j=1,\dotsc,n$, be the event that exactly one random point, say $Y_0,\dotsc,Y_d$, is contained in each of the sets $D_i'(x_j)$, i.e., $Y_i\in D_i'(x_j)$, $i=0,\dotsc,d$, and every other point is outside of $C^K(x_j,h_n)\cap \bd K$, i.e., $Y_i\not\in C^K(x_j,h_n)\cap \bd K$ for $i=d+1,\dotsc, n-1$. Then,
	\begin{align*}
	\PP(A_j) 
	&=    \binom{n}{d+1} \, \PP(Y_i \not\in C^K(x_j,h_n)\cap \bd K, i\geq d+1)\, \PP( Y_i\in D_i'(x_j), i=0,\dotsc,d )\\
	&\geq \binom{n}{d+1} \, (1-\sigma(C^K(x_j,h_n)\cap \bd K))^{n-d-1} \,\prod_{i=0}^d \sigma(D_i'(x_j))\\
	&\geq \binom{n}{d+1} \, \left(1-\frac{c_8}{n}\right)^{n-d-1} c_7^{\frac{d^2-1}{2}} (c_4n)^{-d-1}.
	\end{align*}
	As a consequence, there is $c_9>0$ such that for all $n$ large enough and all $j=1,\dotsc,n$, we have that
	\begin{equation*}
	\PP(A_j) \geq c_9 >0,
	\end{equation*}
	which yields
	\begin{equation*}
	\EE \sum_{j=1}^n \mathbf{1}_{A_j} = \sum_{j=1}^n \PP(A_j) \geq c_9n >0.
	\end{equation*}
	
	Next, let $\mathcal{F}$ be the $\sigma$-algebra generated by the positions of all $(Y_1,\dotsc,Y_n)$ except those which are contained in $D_0'(x_j)$ with $\mathbf{1}_{A_j}= 1$ for at least one $j=1,\dotsc,n$. Hence, if $(Y_1,\dotsc,Y_n)$ is $\mathcal{F}$-measurable and $\mathbf{1}_{A_j}=1$, then, up to reordering, we may assume that $Y_j\in D_0'(x_j)$ is random and $Y_{j+k}=y_k^j \in D_k'(x_j)$ is fixed for $k=1,\dotsc,d$.
	Let $(Y_1,\dotsc,Y_n)$ be an arbitrary $\mathcal{F}$-measurable random vector. If $\mathbf{1}_{A_j}(Y_1,\dotsc,Y_n) =\mathbf{1}_{A_k}(Y_1,\dotsc,Y_n)=1$ for some $j,k\in\{1,\dotsc,n\}$, $j\neq k$, and assuming without loss of generality that $Y_j\in D_0'(x_j)$ and $Y_k\in D_0'(x_k)$, then $Y_j$ and $Y_k$ are vertices of $K_\sigma(n)=[Y_1,\dotsc,Y_n]$ and by \eqref{eqn:independence-cond} it is not possible that there is an edge between $Y_j$ and $Y_k$. Therefore, the change of weighted volume affected by moving $Y_j$ in $D_0'(x_j)$ is independent of the change of the weighted volume of moving $Y_k$ in $D_0'(x_k)$. 
	This yields
	\begin{equation}\label{eqn:independence_cond2}
	\Var[\Phi(K_\sigma(n))|\mathcal{F}] = \sum_{j=1}^n \mathbf{1}_{A_j} \Var_{Y_j} \Phi([Y_j,y_1^j,\dotsc,y_d^j])
	\end{equation}    
	Thus, for large enough $n$, we finally derive from the total variance formula that
	\begin{align*}
        \Var \Phi(K_{\sigma}(n)) 
        &= \EE\Var [\Phi(K_\sigma(n))|\mathcal{F}] + \Var \EE[\Phi(K_\sigma(n))|\mathcal{F}] \\
        &\geq \EE \Var[\Phi(K_{\sigma}(n))|\mathcal{F}]\\
        &= \EE \sum_{j=1}^n \mathbf{1}_{A_j} \Var_{Y_j}\Phi([Y_j,y_1^j,\dotsc,y_d^j]) \\
        &\geq \frac{1}{c_5} (h_n)^{d+1}\, \EE \sum_{j=1}^n \mathbf{1}_{A_j}
        \geq c n^{-\frac{2(d+1)}{d-1}+1} 
        = cn^{-\frac{d+3}{d-1}},
	\end{align*}
	where $c>0$ is some constant. This completes the prove of Theorem \ref{thm:var-bound}.
\end{proof}

\subsection{Proof of the main theorem}\label{subsec:proof_main}
In this section we prove Theorem \ref{thm:CLT}. To do so, we apply Lemma \ref{lem:CLTLachiezeReyPeccati} to the random variable $$W(n) = \frac{ \Phi(K_\sigma(n)  ) - \EE \Phi(K_\sigma(n) )}{\sqrt{\Var \Phi(K_\sigma(n) ) }} $$ and deduce that $d_{\Kol} (W(n), Z) \rightarrow 0$ as $n\to \infty$, where $Z$ is a standard Gaussian random variable. To apply the lemma, we consider a vector $X = (X_1, \ldots, X_n)$ of independent points on $\bd K$ distributed according to $\sigma$. We set $\tilf(X) = \Phi([X_1, \ldots, X_n])$ and $f(X) = \frac{\tilf (X) -\EE \tilf(X)}{\sqrt{\Var \tilf (X)} }$, and note that by definition $f(X) = W(n)$. Note moreover that $f$ and $\tilf$ may be extended in an obvious manner to symmetric functions on $\bigcup_{k=1}^n(\bd K)^k$.

Our estimation of the first- and second-order difference operators is based on the following simple observation. For a point $z \in \bd K$ and a convex $L \subset K$ we set $\Delta(z, L) := {\rm conv} (L \cup \{x\}) \setminus L$. Then, if the $\sigma$-surface body $K_\sigma^{\tau}$ is contained in $[X_2, \ldots, X_n]$, one has 
\begin{equation}\label{eq:observartion}
\Delta(X_1, [X_2, \ldots, X_n]) \subset \Vis_\sigma(X_1, \tau).
\end{equation}

We first bound the term $B_{3}(f)$ involving only the first-order difference operator.

\begin{lemma}\label{lem:bound-B3}
	There exists a constant $C=C(K, \Phi, \sigma)>0$ such that 
	\begin{align*}
	B_{3}(f) & \leq C\, n^{-2} \, (\log n)^{4\frac{d+1}{d-1}}. %\\
	% \intertext{and}
	%B_{3,\Wass}(f) & \leq C  \frac{(\log n)^{3\frac{d+1}{d-1}}}{n^{\frac{3}{2}}}
	\end{align*}
\end{lemma}

In the proof below and subsequent ones, the letter $C$ stands for an arbitrary constant (independent of $n$), whose value may change from line to line. As we explained above, $C$ is allowed to depend on $K$, $\Phi$ and $\sigma$.

\begin{proof}
	Note that $D_jf(X) = D_j \tilf (X) \Big/ \sqrt{\Var \tilf(X)}$, and hence $B_{3}(f) = B_{3}(\tilf) \big/ \Var \tilf(X)^2$.
	Therefore, we begin by estimating the term $B_{3}(\tilf)$.  Note that by definition, 
	\begin{equation*}
	D_1\tilf(X) = \Phi(\Delta(X_1, [X_2, \ldots, X_n] ) ).
	\end{equation*}
	Fix $\alpha > 0$, which will be specified later. By Lemma \ref{lem-RVW-surface-body}, there exists a constant $c(\alpha)> 0$ such that, denoting $\tau = c(\alpha) \frac{\log (n-1)}{n-1}$, the event $A := \{K_\sigma^{\tau} \subset [X_2, \ldots, X_n] \}$ has $\PP (A^c) \leq (n-1)^{-\alpha} \leq C n^{-\alpha}$ (where $C$ depends on $\alpha$, for example one can take $C=2^\alpha$). On $A$, we use our observation \eqref{eq:observartion} and the bound \eqref{eq:bound-Phi-Vis}  to obtain
	\begin{equation}\label{eq:bound-D1-Vis}
	|D_1\tilf(X)| \leq \Phi(\Vis_\sigma(X_1, \tau)) \leq C \tau^{\frac{d+1}{d-1}} \leq C \left(\frac{\log n}{n} \right)^{\frac{d+1}{d-1}}.
	\end{equation}
	On $A^c$ we have the trivial bound 
	\begin{equation}\label{eq:bound-D1-trivial}
	|D_1 \tilf(X)| \leq \Phi(K) =: D. 
	\end{equation}
	Combining these estimates with the convexity of the function $t \mapsto t^4$ we find that %$t \mapsto t^p$ (for $p=3,4$) we obtain
	%\begin{align*}
	%\EE   |D_1\tilf(X) |^p  &= \EE   | \id_A \cdot  D_1\tilf(X) + \id_{A^c} \cdot  D_1\tilf(X)  |^p   \\
	%& \leq 2^{p-1} \cdot \EE \left[   \id_A   \cdot  C \left(\frac{\log n}{n} \right)^{p\frac{d+1}{d-1}}  +  \id_{A^c} \cdot  D^p \right] \\
	%& \leq 2^{p-1} \cdot \left[    C \left(\frac{\log n}{n} \right)^{p\frac{d+1}{d-1}}   + D^p n^{-\alpha}    \right]
	%\end{align*}
	\begin{align*}
	\EE   |D_1\tilf(X) |^4  &= \EE   | \id_A \cdot  D_1\tilf(X) + \id_{A^c} \cdot  D_1\tilf(X)  |^4   \\
	& \leq 2^{3} \, \EE \left[   \id_A   \cdot  C \left(\frac{\log n}{n} \right)^{4\frac{d+1}{d-1}}  +  \id_{A^c} \cdot  D^4 \right] \\
	& \leq 2^{3} \, \left[    C \left(\frac{\log n}{n} \right)^{4\frac{d+1}{d-1}}   + D^4 n^{-\alpha}    \right]
	\end{align*}
	Choosing now $\alpha=13 > 4\frac{d+1}{d-1}$, we derive
	\begin{align*}
	%B_{3, \Wass} (\tilf ) &= \EE |D_1\tilf(X)|^3  \leq C \left(\frac{\log n}{n} \right)^{3\frac{d+1}{d-1}} \\
	%\intertext{and}
	B_{3} (\tilf ) &= \EE |D_1\tilf(X)|^4  \leq C \left(\frac{\log n}{n} \right)^{4\frac{d+1}{d-1}}.
	\end{align*}
	Finally, since by Theorem \ref{thm:var-bound} we have that $\Var \tilf(X) \geq C n^{-\frac{d+3}{d-1}}$, we conclude
	\begin{align*}
	B_{3}(f) &= \frac{B_{3}(\tilf)}{ \Var \tilf(X)^2 } \leq C  \left(\frac{\log n}{n} \right)^{4\frac{d+1}{d-1}} \, n^{2\frac{d+3}{d-1}} = C\, n^{-2}\,(\log n)^{4\frac{d+1}{d-1}}.
	%\intertext{and}
	%B_{3, \Wass} &=  \frac{B_{3,\Wass}(\tilf) }{\Var \tilf(X)^{3/2} } \leq C \left(\frac{\log n}{n} \right)^{3\frac{d+1}{d-1}} \cdot n^{\frac{3}{2} \, \frac{d+3}{d-1}} = C \, \frac{(\log n)^{3\frac{d+1}{d-1} } }{n^{\frac{3}{2} } }.  
	\end{align*}
	This completes the argument.
\end{proof}

Next, we turn to the terms $B_1(f)$ and $B_2(f)$, involving the second-order difference operator as well. 

\begin{lemma}\label{lem:boubd-B12}
	There exists a constant $C=C(K, \Phi, \sigma)>0$ such that
	\begin{align*}
	B_1(f) \leq C\, n^{-4}\, (\log n)^{4\frac{d+1}{d-1} + 2  } \qquad\text{and}\qquad
	B_2(f) \leq C\, n^{-3}\, (\log n)^{4\frac{d+1}{d-1} + 1  }.
	\end{align*}
\end{lemma}

\begin{proof}
	First we note that, as before, for $j=1, 2$ one has
	$B_j(f) = B_j(\tilf) \Big / \Var \tilf(X)^2$.
	Therefore we begin with estimating the terms $B_j(\tilf)$.  We note that $D_{1,2}\tilf(Y)=0$ whenever the regions $\Delta(Y_1, [Y_3, \ldots, Y_n])$ and $\Delta(Y_2, [Y_3, \ldots, Y_n])$ are disjoint. We consider this time the event 
	\begin{equation*}
	A' = \left\{ K_\sigma^{\tau} \subset \bigcap_{W \in \{Y, Y', Z, Z'\} } [W_4, \ldots, W_n]    \right\}.
	\end{equation*}
	Using Lemma \ref{lem-RVW-surface-body} (along with the union bound), for a fixed $\alpha > 0$ (to be specified later), one can find $c(\alpha)> 0$ such that, for $\tau = c(\alpha) \frac{\log n}{n}$, one has $\PP((A')^c) \leq C n^{-\alpha}$. On $A'$, our observation \eqref{eq:observartion} implies that
	\begin{equation*}
	\id\{D_{1,2}\tilf(Y) \neq 0 \} \leq \id \{ \Vis_\sigma(Y_1, \tau) \cap \Vis_\sigma(Y_2, \tau) \neq \emptyset  \},
	\end{equation*}
	and the analogous statement holds for $D_{1,3}f(Y')$. Combined with the bound \eqref{eq:bound-D1-Vis} for the first-order difference operator and the estimate \eqref{eq:bound-Vis-disjoint} (and on $(A')^c$, the trivial bound \eqref{eq:bound-D1-trivial}), this yields
	\begin{align*}
	B_1(\tilf) &= \EE \left[\id\{D_{1,2}f(Y)=0 \} \id\{D_{1,3}f(Y')=0\}\, |D_2f(Z)|^2\, |D_3f(Z')|^2 \right]  \\
	& \leq \EE \left[  \id_{A'} \id \{ \Vis_\sigma(Y_1, \tau) \cap \Vis_\sigma(Y_2, \tau)  \}  \id \{ \Vis_\sigma(Y_1', \tau) \cap \Vis_\sigma(Y_3', \tau)  \}\,  C \left(\frac{\log n}{n}\right)^{4\frac{d+1}{d-1}} \right] \\
	&\hspace{8cm}+ \EE \left[\id_{(A')^c} D^2 \right] \\
	& \leq C \left(\frac{\log n}{n}\right)^{2} \left(\frac{\log n}{n}\right)^{4\frac{d+1}{d-1}} + C D^2 n^{-\alpha},
	\end{align*}
	where in the last step we used the independence of $Y$ and $Y'$ and the fact that the event $A'$ depends only on the entries $W_j$ for $j \geq 4$. Finally, picking $\alpha=6 > 2 + \frac{d+1}{d-1}$, we derive that
	\begin{equation*}
	B_1(\tilf ) \leq C\left(\frac{\log n}{n}\right)^{4\frac{d+1}{d-1} + 2}.
	\end{equation*}
	Next, a very similar computation yields the estimate
	\begin{equation*}
	B_2(\tilf) \leq C \left(\frac{\log n}{n}\right)^{4\frac{d+1}{d-1} + 1}.
	\end{equation*}
	Finally, using again the bound $\Var \tilf(X) \geq C n^{-\frac{d+3}{d-1}}$  provided by Theorem \ref{thm:var-bound}, we conclude that
	\begin{align*}
	B_1(f) &= \frac{B_1(\tilf)}{\Var \tilf(X)^2} \leq  C \left(\frac{\log n}{n}\right)^{4\frac{d+1}{d-1} + 2}  \, n^{2\frac{d+3}{d-1}} 
        = C\, n^{-4}\,  (\log n)^{4\frac{d+1}{d-1} +2  } \\ \intertext{and}
	B_2(f) & = \frac{B_2(\tilf)}{\Var \tilf(X)^2} \leq  C   \left(\frac{\log n}{n}\right)^{4\frac{d+1}{d-1} + 1}  \,  n^{2\frac{d+3}{d-1}}  
        = C\, n^{-3}\, (\log n)^{4\frac{d+1}{d-1} + 1  }.
	\end{align*}
	This completes the proof.
\end{proof}

With these estimates established, we can now prove the asymptotic normality for the weighted volume of random weighted inscribed polytopes.

\begin{proof}[Proof of Theorem \ref{thm:CLT}]
	Let, as above, $W(n) = f(X) = \frac{\Phi(K_\sigma(n)) -\EE \Phi(K_\sigma(n)) }{\sqrt{\Var\Phi(K_\sigma(n))}}$, and observe that $\EE W(n) = 0$ and $\Var W(n)=1$. Applying Lemma \ref{lem:CLTLachiezeReyPeccati}, and using the bounds provided by Lemmas \ref{lem:boubd-B12} and \ref{lem:bound-B3}, we derive
	\begin{align*}
	d_{\Kol}(W(n), Z) &\leq C \sqrt{n} \left[ n \sqrt{B_1(f)}  + \sqrt{n B_2(f)}  + \sqrt{B_{3}(f)} \right] \\
	& \leq C \sqrt{n} \left[ n \, \frac{(\log n)^{2\frac{d+1}{d-1}+1 } }{n^2} + \sqrt{n} \, \frac{(\log n)^{2\frac{d+1}{d-1}+\frac{1}{2} } }{n^{\frac{3}{2}}} + \frac{(\log n)^{2\frac{d+1}{d-1} } }{n} \right] \\
	& \leq C\, n^{-\frac{1}{2}}\, (\log n)^{2\frac{d+1}{d-1}+1 }, 
	%	\intertext{and}
	%	d_{\Wass}(W(n), Z) &\leq C \sqrt{n} \left[ n \sqrt{B_1(f)}  + \sqrt{n B_2(f)}  + \sqrt{B_{3,\Wass}(f)} \right] \\
	%	& \leq C \sqrt{n} \left[ n \, \frac{(\log n)^{2\frac{d+1}{d-1}+1 } }{n^2} + \sqrt{n} \, \frac{(\log n)^{2\frac{d+1}{d-1}+\frac{1}{2} } }{n^{\frac{3}{2}}} + \frac{(\log n)^{\frac{3}{2}\,\frac{d+1}{d-1} } }{n^{\frac{3}{4}}} \right] \\
	%	& \leq C \frac{(\log n)^{2\frac{d+1}{d-1}+1 } }{n^{\frac{1}{4}}}.
	\end{align*}
	where, as before, $Z$ is a standard Gaussian random variable. In particular, since the last expression tend to zero as $n\to\infty$, this implies convergence in distribution of $W(n)$ to $Z$.
\end{proof}

\section{Proofs of other results}\label{sec:proof_others}

\subsection{Random inscribed polytopes in projective Riemannian geometries}

\begin{proof}[Proof of Theorem \ref{cor:Riemann}]
We fix a Euclidean structure on $\Omega$, with associated $d$-dimensional Lebesgue measure and $(d-1)$-dimensional Hausdorff measure on $\bd K$.  We have to verify that $\sigma_g$ and $\Phi_g$ meet the conditions of Theorem \ref{thm:CLT}. Indeed, by the uniqueness of the Riemannian volume measure (see Section \ref{subsec:Riemannian}), the Euclidean Lebesgue  measure on $K$ and $(d-1)$-dimensional Hausdorff measure on $\bd K$ can be considered as the Riemannian volume measures on $K$ and $\bd K$, respectively,  associated with the Euclidean structure. As the local expression \eqref{eq:riem-vol-density} shows,  both $\sigma_g$ and the $(d-1)$-dimensional Hausdorff measure on $\bd K$ are given by integrating $C^1$-volume densities, and hence (as the space of volume densities is one-dimensional) they differ by a positive $C^1$-function. A similar reasoning applies to $\Phi_g$ and the Lebesgue measure on $K$.  Therefore, Theorem \ref{thm:CLT} applies here, and proves the result.
\end{proof}

\subsection{Random inscribed polytopes in projective Finsler metrics}

\begin{proof}[Proof of Theorem \ref{thm:Finsler}]
We fix a Euclidean structure on $\Omega$, with an associated $d$-dimensional Lebesgue measure and $(d-1)$-dimensional Hausdorff measure on $\bd K$. We have to verify that $\Phi$ and $\sigma$ satisfy the assumptions of Theorem \ref{thm:CLT}. Indeed, by definition $\sigma$, as well as the $(d-1)$-dimensional Hausdorff measure on $\bd K$, are given as integrals of continuous volume densities. Since the space of densities on $T_x\bd K$ is one-dimensional for all $x\in\bd K$, the two volume densities differ by multiplication by a positive continuous function. The same applies to $\Phi$ and the Lebesgue measure on $K$. Therefore, we can apply Theorem \ref{thm:CLT} and deduce asymptotic normality of  $\Phi(K_F(n))$.
\end{proof}

\subsection{Dual Brunn--Minkowski theory}%\label{sec:dual-BM}

In what follows we will require the following adaptation of Theorem \ref{thm:CLT}. We keep the assumptions of that theorem, and let $T $ be a fixed convex body strictly contained in $K$ and such that $T$ contains the origin in the interior. We use the notation $ K_{\sigma,T}(n)$ for the convex hull of $K_\sigma(n)$ and a $T$. Then we claim that the $\Phi$-measure of $K_{\sigma, T}(n)$ satisfies a central limit theorem, that is,
\begin{equation}\label{eq:modified_CLT}
\frac{\Phi(K_{\sigma,T}(n)) - \EE \Phi(K_{\sigma,T}(n)) }{ \sqrt{\Var \Phi(K_{\sigma,T}(n)) }} \overset{d}{\longrightarrow} Z,
\end{equation}
as $n\to\infty$, where $Z$ is a standard Gaussian random variable. The adaptation of the proof of Theorem \ref{thm:CLT} to this case is rather minor; it suffices to note that for small enough $t>0$, $T$ is contained in the $\sigma$-surface body $K_\sigma^t$, and hence, in view of Lemma \ref{lem-RVW-surface-body}, with overwhelming probability, $K_{\sigma,T}(n) = K_\sigma(n)$. We use this adaptation to prove Theorem  \ref{thm:dual-volume}.

\begin{proof}[Proof of Theorem \ref{thm:dual-volume}]
	A simple integration in polar coordination using formula \eqref{eq:dual_volumes} gives
	\begin{equation*}
	\tV_j(A) = \begin{cases}
	\displaystyle \frac{j}{d} \int_A \|x\|^{j-d}\, dx, & j > 0, \vspace{5pt} \\ 
	\displaystyle \frac{|j|}{d} \int_{\RR^d \setminus A} \|x\|^{j-d}\, dx, & j < 0,
	\end{cases}
	\end{equation*}
	for a convex body $A\subset\RR^d$ containing the origin,
	where
%	, $\| \cdot \|$ denotes the Euclidean norm, and 
	$dx$ indicates integration with respect to the Lebesgue measure on $\RR^d$.
	
	This formula brings the dual volumes into the framework of Theorem \ref{thm:CLT}, with the caveat that for $j < 0$ the density function $\|x\|^{j-d}$ is not integrable at the origin. For $j >0$, however, the density function  $\phi_j(x)=\frac{j}{d} \|x\|^{j-d}$ is integrable on $K$ and continuous near $\bd K$, and $\tV_j(K_{\sigma,T}(n) ) = \Phi_j(K_{\sigma, T} (n))$. For $j < 0$ we note that by definition $K_{\sigma,T}(n)$ contains $T$, so we can get around the problem by taking a measure $\Phi_j$ with Lebesgue density $\phi_j$, such that, on $ \RR^d \setminus T$, $\phi_j(x)=\frac{|j|}{d} \|x\|^{j-d}$, and $\phi_j$ is continuous and positive on $T$. With this definition we have $\tV_j(K_{\sigma,T}(n) ) = \Phi_j(\RR^d) -\Phi_j(K_{\sigma, T} (n))$. Therefore, for any $j \neq 0$, the result follows immediately from the modification \eqref{eq:modified_CLT} of Theorem \ref{thm:CLT}.
	%\DR{Here I changed a bit the proof (to say that the discontinuity at the origin is not a problem for $j>0$).}
\end{proof}

\subsection{Random polyhedral sets}

\begin{proof}[Proof of Theorem \ref{thm:width}]
	We may assume without loss of generality that $K$ contains the origin. First, we note that by \cite[Lemma 2]{Lut75}, for a convex body $A$ containing the origin, $W(L) =  \frac{2}{\kappa_d} \tV_{-1}(L^*) $, where $\tV_{-1}$ denotes the dual volume considered in Section \ref{sec:dual-BM}, and $L^*$ denotes the polar body of $L$, namely
	\begin{equation*}
	L^* = \{ y \in \RR^d \,:\, \forall x \in L  \, \langle x,y \rangle \leq 1 \}.
	\end{equation*}  
	Therefore, denoting $C_d := \frac{2}{\kappa_d} $ we find that
	\begin{equation}\label{eq:W-duality}
        W(P_\sigma(n) \cap L) = C_d \tV_{-1}(\conv (P_\sigma(n)^* \cup L^*) ). 
	\end{equation}
	Consider the \emph{Legendre transform} $\Lambda : \bd K \to \bd K^*$, which assigns to $x \in \bd K$ the unique point $\Lambda(x) \in \bd K^*$ which is proportional to the outer normal to $\bd K$ at $x$. Since by assumption $K$ is of class $C^2_+$, $\Lambda$ is a diffeomorphism. Then,
	\begin{equation*}%\label{eq:P-dual}
	P_\sigma(n)^* = \left(H^-(X_1) \cap \cdots \cap H^-(X_n) \right)^* = [\Lambda(X_1), \ldots, \Lambda(X_n)].
	\end{equation*}
	As the $X_i$ are independent and distributed according to $\sigma$ on $\bd K$, the points $\Lambda(X_i)$ are independent and distributed according to the push-forward measure $\sigma^* := \Lambda_*\sigma$ on $\bd K^*$. Note that $\sigma^*$ has a continuous and positive  density with respect to the $(d-1)$-dimensional Hausdorff measure on $\bd K^*$. Explicitly, according to  \cite[equ.\ 52]{BorReit}, if $\sigma$ has density $\varsigma$, $\sigma^*$ has density $\kappa(\Lambda^{-1}(y))  \varsigma (\Lambda^{-1}(y))\,  \frac{\langle y, n_{K^*}(y)\rangle}{\|y\|^d}$, where $\kappa$ is the Gauss--Kronecker curvature of $\bd K$
%	,$\langle\, \cdot\,, \,\cdot\, \rangle$ and $\|\,\cdot\,\|$  denote the Euclidean inner product and norm, respectively, 
	and $n_{K^*}(y)$ is the unit normal vector to $\bd K^*$ at $y$.
	
	In other words, $P_\sigma(n)^* = [\Lambda(X_1), \ldots, \Lambda(X_n)]$ is equal in distribution to $K^*_{\sigma^*}(n)$, the random polytope inscribed in $K^*$ generated by $n$ independent random points with distribution $\sigma^*$ on $\bd K^*$, and therefore, in the notation of Theorem \ref{thm:dual-volume}, $ \conv (P_\sigma(n)^* \cup L^*) $ is equal in distribution to $ K^*_{\sigma^*, L^*}(n)$. Combining this with \eqref{eq:W-duality}, we find that the random variables $W(P_\sigma(n) \cap L)$ and $C_d \tV_{-1}(K^*_{\sigma^*, L^*}(n) )$ are equal in distribution. The result now follows from Theorem \ref{thm:dual-volume}. 
\end{proof}

\subsection*{Acknowledgement}
DR and CT were supported by the German Research Foundation (DFG) via CRC/TRR 191 \textit{Symplectic Structures in Geometry, Algebra and Dynamics}.

\end{document}